\documentclass[11p]{article}
\usepackage[a4paper, total={6in, 9in}]{geometry}

\usepackage[utf8]{inputenc}
\usepackage[T1]{fontenc}
\usepackage{lmodern}
\usepackage{amsthm}
\usepackage{comment}
\usepackage{bm,amsfonts,nicefrac,latexsym,amsmath,amsfonts,amsbsy,amscd,amsxtra,amsgen,amsopn,bbm,amssymb}

\newtheorem{lemma}{Lemma}
\newtheorem{theorem}{Theorem}

\usepackage{lastpage}
\usepackage{fancyhdr}
\usepackage{mathrsfs}
\usepackage[usenames,dvipsnames]{xcolor}
\usepackage{graphicx}
\usepackage{enumitem}
\usepackage{listings}
\usepackage{float}
\usepackage[normalem]{ulem} 
\usepackage{textpos}
\usepackage[numbers]{natbib} 

\usepackage[ruled,vlined,linesnumbered]{algorithm2e}
\usepackage{algpseudocode}
\usepackage{algorithm2e}

\usepackage{tabularx,longtable,multirow,caption}
\usepackage[english]{babel}
\usepackage{dsfont}
\usepackage{epstopdf} 
\usepackage{array}
\usepackage[pdftex,hypertexnames=false,colorlinks]{hyperref} 
\usepackage[colorinlistoftodos]{todonotes}

\allowdisplaybreaks
\hypersetup{pdftitle={},
	pdfsubject={}, 
	pdfauthor={},
	pdfkeywords={}, 
	pdfstartview=FitH,
	pdfpagemode={UseOutlines},
	bookmarksnumbered=true, bookmarksopen=true, colorlinks,
	citecolor=black,%
	filecolor=black,%
	linkcolor=black,%
	urlcolor=black}

\usepackage{booktabs}
\usepackage{pgfplots}
\usepackage{pgfplotstable} 
\usepackage{pdflscape}
\usepackage{rotating}
\usepackage{siunitx} 
\usepackage{dcolumn}
\usepackage{scalerel}
\usepackage{tikz}
\usetikzlibrary{svg.path}
\pgfplotstableset{
	begin table=\begin{longtable},
		end table=\end{longtable},
}

\pgfplotsset{compat=1.17}

\DeclareMathOperator*{\argmin}{arg\,min}

\allowdisplaybreaks

\definecolor{orcidlogocol}{HTML}{A6CE39}
\tikzset{
	orcidlogo/.pic={
		\fill[orcidlogocol] svg{M256,128c0,70.7-57.3,128-128,128C57.3,256,0,198.7,0,128C0,57.3,57.3,0,128,0C198.7,0,256,57.3,256,128z};
		\fill[white] svg{M86.3,186.2H70.9V79.1h15.4v48.4V186.2z}
		svg{M108.9,79.1h41.6c39.6,0,57,28.3,57,53.6c0,27.5-21.5,53.6-56.8,53.6h-41.8V79.1z M124.3,172.4h24.5c34.9,0,42.9-26.5,42.9-39.7c0-21.5-13.7-39.7-43.7-39.7h-23.7V172.4z}
		svg{M88.7,56.8c0,5.5-4.5,10.1-10.1,10.1c-5.6,0-10.1-4.6-10.1-10.1c0-5.6,4.5-10.1,10.1-10.1C84.2,46.7,88.7,51.3,88.7,56.8z};
	}
}
\newcommand\orcidicon[1]{\href{https://orcid.org/#1}{\mbox{\scalerel*{
				\begin{tikzpicture}[yscale=-1,transform shape]
				\pic{orcidlogo};
				\end{tikzpicture}
			}{|}}}}
\usepackage{hyperref}

\hypersetup{%
  colorlinks=true,
  linkcolor=blue,
  linkbordercolor={0 0 1}
}

\newcommand{\R}{\mathbb{R}}
\newcommand{\sn}{{\mathcal S}^n}
\newcommand{\inprod}[2]{{\langle #1,#2 \rangle}} 
\newcommand{\trace}{\textrm{tr}}
\newcommand{\rank}{\textrm{rank}}
\newcommand{\diag}{\textrm{diag}}
\newcommand{\Diag}{\textrm{Diag}}
\newcommand{\vect}{\textrm{vec}}
\newcommand{\Col}{\textrm{Col}}
\newcommand{\Nul}{\textrm{Nul}}
\newcommand{\Conv}{\textrm{Conv}}
\newcommand{\epsadmm}{\varepsilon_\textrm{ADMM}}
\newcommand{\epsproj}{\varepsilon_\textrm{proj}}
\newcommand{\inner}{\text{inner}}
\newcommand{\arrow}{\text{arrow}}

\newcounter{rowcntr}[table]
\renewcommand{\therowcntr}{\thetable.\arabic{rowcntr}}
\newcolumntype{N}{>{\refstepcounter{rowcntr}\therowcntr}c}
\AtBeginEnvironment{tabular}{\setcounter{rowcntr}{0}}

\providecommand{\keywords}[1]
{
  \small	
  \textbf{\textit{Keywords---}} #1
}
\usepackage[multiple]{footmisc}
\makeatletter
\newcommand\footnoteref[1]{\protected@xdef\@thefnmark{\ref{#1}}\@footnotemark}
\makeatother

\title{Partitioning through projections: strong SDP bounds for large graph partition problems}
\makeatletter
\newcommand\newtag[2]{#1\def\@currentlabel{#1}\label{#2}}
\makeatother

\date{ }

\begin{document}

\author{
Frank de Meijer\footnote{Tilburg University, Department of Econometrics \& Operations Research, CentER, 5000 LE Tilburg,
\href{mailto:F.J.J.deMeijer@tilburguniversity.edu}{f.j.j.demeijer@tilburguniversity.edu},
\href{mailto:R.Sotirov@tilburguniversity.edu}{r.sotirov@tilburguniversity.edu}}~\footnote{Corresponding Author: \href{mailto:F.J.J.deMeijer@tilburguniversity.edu}{f.j.j.demeijer@tilburguniversity.edu}}~\orcidicon{0000-0002-1910-0699} \, \,\, \and
Renata Sotirov$^*$\orcidicon{0000-0002-3298-7255
} \and Angelika Wiegele\footnote{Institut f\"ur Mathematik, Alpen-Adria-Universit\"at Klagenfurt, Universit\"atstra{\ss}e 65-67, 9020 Klagenfurt, \href{mailto:angelika.wiegele@aau.at}{angelika.wiegele@aau.at}, \href{mailto:shudian.zhao@aau.at}{shudian.zhao@aau.at}}~~\footnote{This project has received funding from the European Union’s Horizon 2020 research and innovation programme under the Marie Sk\l{}odowska-Curie grant agreement MINOA No 764759.}~\orcidicon{0000-0003-1670-7951}
			 \and Shudian Zhao$^{\ddagger\,\S}$\orcidicon{0000-0001-6352-0968} 
}

\maketitle
\begin{abstract}
The graph partition problem (GPP) aims at clustering  the  vertex set of a graph into a fixed number of disjoint  subsets of given sizes such that  the sum of weights of edges joining different sets is minimized.
This paper investigates the quality of doubly nonnegative (DNN) relaxations, i.e., relaxations having matrix variables that are both positive semidefinite and nonnegative, strengthened by \textcolor{black}{additional} polyhedral cuts for two variations of the GPP: the $k$-equipartition and the graph bisection problem.
After reducing the size of  the relaxations by facial reduction, we solve them by a cutting-plane algorithm that combines  an augmented Lagrangian method with Dykstra’s projection algorithm. 
Since many components of our algorithm are  general, the algorithm is suitable for solving various DNN relaxations  with a large number of cutting planes.

We are the first to show the power of DNN relaxations with  additional cutting planes for the GPP on large benchmark instances up to 1,024 vertices. Computational results show impressive improvements in strengthened DNN bounds.
\\

 \end{abstract}

\section{Introduction}
The graph partition problem (GPP) is the problem of partitioning the vertex set of a graph into a fixed number of subsets, say $k$,  of given sizes such that the sum of weights of edges joining different sets is minimized. In the case that all sets are of equal sizes we refer to the resulting GPP as the $k$-equipartition problem ($k$-EP).  The case of the graph partition problem with $k=2$ is known as the graph bisection problem (GBP). In the GBP the sizes of two subsets might differ. 
The special case of the GBP where both subsets have the same size is known in the literature as the equicut problem, see e.g.~\cite{KarischRendlClausen}.

The graph partition problem is known to be NP-hard  \cite{Garey1976SomeSN}. It is a fundamental problem that is extensively studied, mostly due to its applications in numerous fields, including  VLSI design, social networks, floor \textcolor{black}{planning}, data mining, air traffic, image processing, image segmentation, parallel computing and telecommunication, see e.g., the book~\cite{Bichot}  and the references therein. Recent studies in quantum circuit design \cite{quantum} also relate to the general graph partition problem. Furthermore, the GPP is used to compute bounds for the bandwidth problem~\cite{RENDL2021105422}.
For an overview of recent advances in graph partitioning, we refer the reader to~\cite{Beluc}.

There exist bounding approaches for the GPP that are valid for all variations of the GPP. We list some of them in this paragraph.
Donath and Hoffman~\cite{Donath1973LowerBF}  derive an eigenvalue bound for the  graph partition problem. That bound is improved by Rendl and Wolkowicz~\cite{RendlWolkowicz}. Wolkowicz and Zhao~\cite{wokowicz-zhao1999} derive semidefinite programming (SDP) relaxations for the graph partition problem that are based on the so-called vector lifting approach. That is, relaxations  in~\cite{wokowicz-zhao1999} have matrix variables of order $nk+1$, where $n$ is the number of vertices in the graph.
An SDP relaxation  based on the so-called matrix lifting approach is derived in
\cite{sotirov2014gpp}, resulting in a  compact relaxation  for the GPP having matrix variables of order $n$.  The relaxation from \cite{sotirov2014gpp} is a doubly nonnegative (DNN) relaxation, which is an SDP relaxation over the set of nonnegative matrices. In general, vector lifting relaxations provide stronger bounds than matrix lifting relaxations, {\color{black}see e.g.,~\cite{Olga,sotirov2014gpp}.}

For the $k$-equipartition problem, Karisch and Rendl~\cite{Karisch98semidefinite}, among others, show how to reformulate the Donath-Hoffman and the Rendl-Wolkowicz relaxations as semidefinite programming  problems. Karisch and Rendl also present several SDP relaxations with increasing complexity for the $k$-EP with matrix variables of order $n$ that dominate relaxations from~\cite{Donath1973LowerBF,RendlWolkowicz}. The strongest SDP relaxation from~\cite{Karisch98semidefinite} is a DNN relaxation with additional triangle and independent set inequalities. That relaxation provides the best known SDP bounds for the $k$-EP. However, it is  difficult to compute those bounds for general graphs with more than 300 vertices when using interior point methods. 
In~\cite{Dam2015SemidefinitePA}, the authors compute
bounds from~\cite{Karisch98semidefinite} for  highly symmetric graphs by aggregating the triangle and independent set inequalities. Their approach results in significantly reduced   computational effort for solving SDP relaxations with \textcolor{black}{a large number} of cutting planes for highly symmetric graphs.
De~Klerk et al.~\cite{Klerk2012OnSP} exploit the fact that the $k$-EP is a special case of the quadratic assignment problem (QAP) to derive a DNN relaxation for the $k$-EP  from a DNN relaxation for the QAP. The size of the resulting relaxation is reduced  by exploiting \textcolor{black}{symmetry}.  Numerical results show that the relaxation from~\cite{Klerk2012OnSP} does not dominate the strongest relaxation from \cite{Karisch98semidefinite}.

There exist several SDP relaxations for the graph bisection problem.
Two equivalent  SDP relaxations for the GBP whose matrix variables have  order $n$ are given in \cite{KarischRendlClausen,sotirov2014gpp}. The strongest matrix lifting DNN relaxation with a matrix variable of order $n$ is derived in \cite{Sotirov18}.
That relaxation is equivalent to the vector lifting DNN relaxation  from~\cite{wokowicz-zhao1999}, which has matrix variables of order $2n + 1$.
For \textcolor{black}{comparisons} of the above mentioned  SDP relaxations for the $k$-EP and the GBP, see e.g.,~\cite{sotirov2012sdp} and \cite{sotirov2014gpp}, respectively.

Finally, we list below exact methods for solving the graph partition problem and its variants. Karisch et al.~\cite{KarischRendlClausen}  develop a branch-and-bound algorithm that is based on a cutting-plane approach. In particular, the algorithm combines semidefinite programming and polyhedral constraints, i.e., triangle inequalities, to solve instances of the GBP. The algorithm from~\cite{KarischRendlClausen} solves problem instances with up to 90 vertices.
Hager et al.~\cite{Hager2013AnEA} present a branch-and-bound algorithm for the graph bisection problem that exploits continuous quadratic programming formulations of the problem. The use of SDP bounds leads to the best performance of the algorithm from~\cite{Hager2013AnEA}. 
Armbruster et al.~\cite{HelmbergEtAl} evaluate the strength of a branch-and-cut framework on large and sparse instances, using linear and SDP relaxations of the GBP. 
The numerical results in~\cite{HelmbergEtAl} show that in the majority of the cases the semidefinite programming approach outperforms the linear one. 
Since prior to~\cite{HelmbergEtAl} it was widely believed that SDP relaxations are useful only for  small and dense instances,  the results of~\cite{HelmbergEtAl} influenced  recent solver developments.
As observed by \cite{Beluc}, all mentioned exact approaches typically solve only very small problems with very large running times.
The aim of this paper is to further investigate the quality of DNN relaxations on large graphs, that are strengthened by polyhedral cuts and solved by a first order method.

\subsection*{Main results and outline} 

Doubly nonnegative relaxations are known to provide superior  bounds for various optimization problems. Although  additional cutting planes further improve DNN relaxations, it is  extremely difficult to compute the resulting  bounds already for relaxations with matrix variables of order 300 via interior point methods. 
We design  an efficient algorithm for computing DNN bounds with a huge number of additional cutting planes and show the power of the resulting bounds for two variations of the  graph partition problem.

We conduct a study for the $k$-equipartition problem and the graph bisection problem. Although there exists a DNN relaxation for the GPP~\cite{wokowicz-zhao1999} that is suitable for both problems, we study the problems separately. 
Namely, the $k$-EP allows various equivalent SDP relaxations having different sizes of matrix variables due to the problem's invariance under  permutations of the subsets. Since one can solve DNN relaxations with  smaller matrix variables more efficiently than those with larger matrix variables, we consider the matrix lifting DNN relaxation for the $k$-EP from~\cite{Karisch98semidefinite} that is strengthened by the triangle and independent set inequalities.
On the other hand, the vector lifting DNN relaxation for the GBP from~\cite{wokowicz-zhao1999} is known to dominate matrix lifting DNN relaxations for the same problem. Therefore, we consider the vector lifting DNN relaxation for the GBP and strengthen it by adding boolean quadric polytope (BQP) inequalities.

Prior to solving the DNN relaxations, we use facial reduction to obtain equivalent smaller dimensional relaxations that are strictly feasible. The approach we use for the GBP is based on the dimension of the underlying polytope.
{\color{black}Although strict feasibility of an SDP is not required for our solver, it makes the procedure more efficient.}

{\color{black}
To solve the DNN relaxations with additional polyhedral inequalities, we design a cutting-plane algorithm called the cutting-plane ADMM (CP--ADMM) algorithm. Our algorithm combines the Alternating Direction Method of Multipliers (ADMM) with Dykstra’s projection algorithm. 
The ADMM exploits the natural splitting of the relaxations that arises from the facial reduction.
Dykstra's cyclic projection algorithm finds projections  onto polyhedra induced by the violated cuts.
Since facial reduction eliminates redundant constraints and projects a relaxation
onto a smaller dimensional space,  the projections in the CP--ADMM are easier and faster. 
To further improve efficiency of the CP--ADMM, we cluster non-overlapping cuts, which allows us to perform the projections in each cluster simultaneously. 
Efficiency of the algorithm is also due to the exploitation of warm starts, as well as the use efficient separation routines.
Since we present the various components of the CP--ADMM in a general way, the algorithm is suitable for solving various DNN relaxations incorporating additional cutting planes. 
}

Our numerical results show that the CP--ADMM algorithm  computes strong GPP bounds for graphs with up to 1,024 vertices by adding at most 50,000 cuts in less than two hours. 
Since our algorithm does not require lots of memory, we are able to compute strong bounds for even larger graphs than presented here.
Numerical results also show that the additional cutting planes significantly improve the DNN relaxations, and that the resulting bounds can close gaps for instances with up to 500 vertices.

The paper is structured as follows. In Section~\ref{section:problem} we introduce the graph partition problem. Section~\ref{sect:equipartition} presents DNN relaxations for the $k$-EP. In Section~\ref{sect:bisection} we present DNN relaxations for the GBP and show how to apply facial reduction by exploiting the dimension of the bisection polytope. Our cutting-plane augmented Lagrangian algorithm is introduced in Section~\ref{sect:admm}. The main ingredients of the algorithm are given in Section~\ref{sect:ADMMbasic} and Section~\ref{sect:Dykstra}.  In particular,  Section~\ref{sect:ADMMbasic} explains steps of the ADMM and  Section~\ref{sect:Dykstra} introduces Dykstra’s projection algorithm and its semi-parallelized version.
The CP--ADMM algorithm is outlined in Section~\ref{sect:CPADMM}. Section~\ref{sect:cuts} considers various families of cutting planes that are used to strengthen DNN relaxations for the $k$-EP and GBP.
Numerical results are given in Section~\ref{sect:numerics}.

\section*{Notation}
The set of $n\times n$ real symmetric matrices is denoted by ${\mathcal S}^n$.
The space of symmetric matrices is considered with the trace inner product, which for any $X, Y \in {\mathcal S}^{n}$ is defined as $\inprod{X}{Y}:= \trace (XY)$. 
The associated norm is the Frobenius norm  $\| X\|_F := \sqrt{\trace (XX)}$.
The cone of symmetric positive semidefinite matrices  of order $n$ is defined as ${\mathcal S}_+^n :=\{X \in  {\mathcal S}^n: X\succeq \mathbf{0} \}$.

We use $[n]$ to denote the set of integers $\{1,\dots,n\}$.
The $\vect(\cdot)$ operator maps an $n \times m$ matrix to a vector of length $nm$ by stacking the columns of the matrix on top of one another.
The Kronecker product $A \otimes B$ of matrices $A \in {\R}^{p \times q}$ and $B\in {\R}^{r\times s}$ is defined as the $pr \times qs$ matrix composed of $pq$ blocks of size $r\times s$, with block $ij$ given by $a_{ij}B$, $i \in [p]$, $j \in [q]$. We use the following properties of the Kronecker product and the $\vect$ operator:
\begin{align} 
\vect (AXB)= (B^\top \otimes A)\vect(X) \label{kron:property1}\\
\trace(AB) = \vect(A^\top)^\top \vect(B). \label{kron:property2}
\end{align}
The Hadamard product of two matrices $X=(x_{ij})$ and $Y=(y_{ij})$ of the same size is denoted by $X\circ Y$
and is defined as $(X\circ Y)_{ij} := x_{ij}y_{ij}$.

We denote by ${\mathbf 1}_n$ the vector of all ones of length $n$, and define ${\mathbf J}_n: = {\mathbf 1}_n {\mathbf 1}_n^\top$. The all zero  vector of length $n$ is denoted by ${\mathbf 0}_n$. We use ${\mathbf I}_n$ to denote the identity matrix of order $n$, while its $i$-th column is given by ${\mathbf u}_i$. In case that the dimension of ${\mathbf 1}_n$, ${\mathbf 0}_n$, ${\mathbf J}_n$ and $\mathbf{I}_n$ is clear from the context, we omit the subscript.

The operator $\textrm{diag}\colon \R^{n\times n} \to \R^n$ 
 maps a square  matrix to a vector consisting of its diagonal elements. Its adjoint operator is denoted by $\textrm{Diag}\colon \R^n \to \R^{n\times n}$.  The rank of matrix ${X}$ is denoted by $\mbox{rank}({X})$.

\section{The graph partition problem} \label{section:problem}

Let $G=(V,E)$ be an undirected graph with vertex set $V$,  $|V|=n$,  and edge set $E$. Let $w\colon E \rightarrow \mathbb{R}$ be a weight function on the edges. Let $k$ be a given integer such that $2\leq k \leq n-1$. The  graph partition problem is to partition the vertex set of $G$ into $k$ disjoint sets $S_1$, \ldots, $S_k$ of specified sizes $m_1\geq \cdots \geq m_k \geq 1$, $\sum_{j=1}^k m_j=n$ such that the total weight of edges joining different sets $S_j$ is minimized. If $k=2$, then we refer to the corresponding graph partition problem as the graph bisection problem.  If $m_1=\cdots = m_k = n/k$, then the resulting GPP is known as the $k$-equipartition problem.

Let $A_w$ denote the weighted adjacency matrix of $G$ with respect to $w$ and ${\mathbf m}=(m_1,\ldots,m_k)^\top$. For a given partition of $G$ into $k$ subsets,  let $P=(P_{ij}) \in \{0,1\}^{n\times k}$ be the partition matrix defined as follows:
\begin{align} \label{Y}
P_{ij} := 
\begin{cases}
1,& i\in S_j \\
0, & \mbox{otherwise}
\end{cases}\quad  i \in [n], ~j\in [k].
\end{align}
Thus, the $j$-th column of $P$ is the characteristic vector of $S_j$.
{\color{black}The total weight of the partition}, i.e.,  the  sum of weights of edges that join different sets equals 
$$
\frac{1}{2}\trace \left( A_w({\mathbf J}_n-PP^\top) \right)  =\frac{1}{2} \trace  (LPP^\top), 
$$
where {$L:= \textrm{Diag}(A_w{\mathbf 1}_n) -A_w$} is the weighted Laplacian matrix of $G$. 
The GPP can be formulated as the following binary optimization problem:
\begin{subequations}\label{binary}
\begin{align}
\min_{P}~&  \frac{1}{2}  \langle L,PP^\top \rangle \label{binary-object} \\
(GPP) \qquad \textrm{s.t.}~& ~P{\mathbf 1}_k = {\mathbf 1}_n, \label{binary-a}\\
& P^\top  {\mathbf 1}_n = {\mathbf m}, \label{binary-b}\\
& P_{ij} \in \{ 0, 1 \}\quad \forall i \in [n], ~j\in [k], \label{binary-d}
\end{align}
\end{subequations}
where $P \in \mathbb{R}^{n\times k}$. {\color{black}Note that the objective function \eqref{binary-object} is quadratic.} The constraints  \eqref{binary-a} ensure that each vertex must be in exactly one subset.
The cardinality constraints \eqref{binary-b} take care that  the number of vertices in  subset $S_j$ is $m_j$ for $j\in [k]$.

\medskip\medskip
Let us briefly consider the polytope induced by all feasible partitions of $G$. Let $F_n^k({\mathbf m})$ be the set of all characteristic vectors representing a partition of $n$ vertices into $k$ disjoint sets corresponding to the cardinalities in ${\mathbf m}$. In other words, $F_n^k({\mathbf m})$ contains binary vectors of the form $\vect(P)$, where $P$ is a partition matrix, see \eqref{Y}:
\begin{align} \label{Def:SetF}
F_n^k({\mathbf m}) = \left\{ x \in \{0,1\}^{kn} \, : \, \, \begin{pmatrix}
\mathbf{I}_k \otimes {\mathbf 1}_n^\top  \\
{\mathbf 1}_k^\top \otimes \mathbf{I}_n
\end{pmatrix} x = \begin{pmatrix}
{\mathbf m} \\ {\mathbf 1}_n
\end{pmatrix} \right\}.
\end{align}
We now define $\Conv(F_n^k({\mathbf m}))$ as the \textit{$k$-partition polytope}. Since the constraint matrix defining $F_n^k({\mathbf m})$ is totally unimodular, this polytope can be explicitly written as follows: 
\begin{align} \label{Def:GPPpolytope}
\Conv(F_n^k({\mathbf m})) = \left\{ x \in \mathbb{R}^{kn} \, : \, \, \begin{pmatrix}
\mathbf{I}_k \otimes {\mathbf 1}_n^\top  \\
 {\mathbf 1}_k^\top \otimes \mathbf{I}_n
\end{pmatrix}  x = \begin{pmatrix}
{\mathbf m} \\ {\mathbf 1}_n
\end{pmatrix}, \, x \geq {\mathbf 0}  \right\}. 
\end{align}
\textcolor{black}{The $k$-partition polytope can be seen as a special case of a transportation polytope, see e.g., \cite{DulmageMendelsohn}. 
We now derive the dimension of the $k$-partition polytope, which will be exploited in Section~\ref{sect:bisection}. The following result is implied by the dimension of the transportation polytope \cite{DulmageMendelsohn}. However, we add a proof for completeness.} 
\begin{theorem} \label{Thm:DimConvF}
The dimension of $\Conv(F_n^k({\mathbf m}))$ equals $(k-1)(n-1)$. 
\end{theorem}
\begin{proof} Let $B := \begin{pmatrix}
\mathbf{I}_k \otimes {\mathbf 1}_n^\top  \\
 {\mathbf 1}_k^\top \otimes \mathbf{I}_n
\end{pmatrix}.$ Since for all $i \in [kn]$ there exists a partition such that $x_i = 1$,
we know that $\dim(\Conv(F_n^k({\mathbf m}))) = \dim(\Nul(B))$. 

Let $b_1, \ldots, b_{kn}$ denote the columns of $B$. It is obvious that the set $\{b_1, \ldots, b_n, b_{n+1},b_{2n+1},\ldots,$ $b_{(k-1)n+1}\}$ is linearly independent. From this it follows that 
$\rank ( B) \geq n +k - 1$. Next, we define for all $l = 1, \ldots, k-1$ and $i = 2, \ldots, n$ a vector $w^{l,i} \in \mathbb{R}^{kn}$ as follows: 
\begin{align*}
(w^{l,i})_j = \begin{cases} +1 & \text{if $j = 1$ or $j = l \cdot n+i$,} \\
-1 & \text{if $j = i$ or $j = l\cdot n+1$,} \\
0 & \text{otherwise.}
\end{cases}
\end{align*}
One can verify that $Bw^{l,i} = \bold{0}$ for all $l = 1, \ldots, k-1$ and $i = 2, \ldots, n$. 
Moreover, since $w^{l,i}$ is the only vector that has a nonzero entry on position $ln + i$ among all defined vectors, the set $\{w^{l,i} \, : \, \, l = 1, \ldots, k-1, i = 2, \ldots, n\}$ is a linearly independent set. This proves that $\dim(\Nul(B)) \geq (k-1)(n-1)$. 

Since $\rank(B) + \dim(\Nul(B)) = kn$, we conclude that $\dim(\Nul(B)) = (k-1)(n - 1)$. 
\end{proof}

\subsection{DNN relaxations for the $k$-equipartition problem} \label{sect:equipartition}

There exist several ways to obtain semidefinite programming relaxations for the $k$-EP.  Namely, to obtain an SDP relaxation for the $k$-EP  one can linearize the objective function of $(GPP)$ by introducing a matrix variable of order $n$, which results in a matrix lifting relaxation, see e.g.,~\cite{Karisch98semidefinite,sotirov2014gpp}.
Another approach is to linearize the objective function by lifting the problem in the space of $(nk+1)\times (nk+1)$ matrices, which results in  a vector lifting relaxation, see~\cite{wokowicz-zhao1999}. We call a DNN relaxation basic if it does not contain additional cutting planes such as triangle inequalities, etc.  It is proven in \cite{sotirov2012sdp}  that the basic matrix and vector lifting DNN relaxations for the $k$-EP are equivalent. A more elegant proof of the same result can be found in  Kuryatnikova et al.~\cite{Olga}.
Since one can solve the basic matrix lifting relaxation from~\cite{Karisch98semidefinite} more efficiently than the equivalent vector lifting relaxation from~\cite{wokowicz-zhao1999}, we develop our algorithm for the matrix lifting relaxation for the $k$-EP.

To linearize the objective from  $(GPP)$ we replace $PP^\top$ by a matrix variable $Y\in \sn$. From \eqref{binary-a} it follows that 
 $Y_{ii} = \sum_{j=1}^{k} P_{ij}^2 = \sum_{j=1}^{k} P_{ij}= 1$ for all $i\in [n]$. From \eqref{binary-a}--\eqref{binary-b} we have $ Y {\mathbf 1}_n = PP^\top {\mathbf 1}_n =  \frac{n}{k}P {\mathbf 1}_k = \frac{n}{k} {\mathbf 1}_n$. After putting those constraints together, adding $Y \geq {\mathbf 0}$ and $Y \succeq {\mathbf 0}$, we arrive at the following DNN relaxation introduced by Karisch and Rendl~\cite{Karisch98semidefinite}:
\begin{equation}  \label{eq:relax_dnn}
(DNN_{EP}) \qquad     
\begin{aligned}
\min_{Y}~ &~~  \frac{1}{2} \langle L,Y \rangle\\
 \textrm{s.t.}~  & \textrm{diag}(Y) = {\mathbf 1}_n,\\
    & Y  {\mathbf 1}_n = \frac{n}{k} {\mathbf 1}_n, \\
    & Y \succeq \mathbf{0}, \quad Y \geq \mathbf{0}.
    \end{aligned}
    \end{equation}
We refer to $(DNN_{EP})$ as the basic matrix lifting relaxation. 
We show below that the nonnegativity constraints in $(DNN_{EP})$   are redundant for the equicut problem.
\begin{lemma} \label{redundantnonegative}
Let $k=2$ and  $Y\in \sn_+$ be such that $\textrm{diag}(Y) = {\mathbf 1}_n$ and $Y{\mathbf 1}_n = \frac{n}{2} {\mathbf 1}_n$. Then $Y \geq {\mathbf 0}$.
\end{lemma}
\begin{proof} 
From  $Y {\mathbf 1}_n = \frac{n}{2} {\mathbf 1}_n$ it follows that ${\mathbf 1}_n$ is an eigenvector of $Y$ corresponding to the eigenvalue $n/2$. Then, the eigenvalue decomposition of $Y$ is
$$Y = \frac{1}{2} {\mathbf J}_n+ \sum_{i=2}^n \lambda_i v_iv_i^\top,$$ 
where $v_i$ is the eigenvector of $Y$ corresponding to the eigenvalue $\lambda_i$ for $i=2,\ldots,n$. Moreover, eigenvectors
$v_i$ are orthogonal to ${\mathbf 1}_n$.
Thus $2Y -{\mathbf J}_n =  2\sum_{i=2}^n \lambda_i v_iv_i^\top \succeq {\mathbf 0}$.

Now, let $Z:= 2Y -{\mathbf J}_n $. From $\mbox{diag}(Y)= {\mathbf 1}_n$  it follows that 
$\mbox{diag}(Z) = {\mathbf 1}_n$. Since  $Z\succeq {\mathbf 0}$ we have that $-1\leq Z_{ij}\leq 1$ for all $i,j\in [n]$, which   implies that $Y_{ij}\geq 0$ for all $i,j \in [n]$.
\end{proof}
{\color{black}For a different proof of Lemma~\ref{redundantnonegative} see e.g., Theorem 4.3 in \cite{Karisch98semidefinite}. }
The relaxation $(DNN_{EP})$ can be further strengthened by adding triangle and independent set inequalities, see Section~\ref{subsect:triangle} and \ref{subsect:clique}, respectively. 
This  strengthened  relaxation is proposed in~\cite{Karisch98semidefinite}  and provides currently the strongest SDP bounds for the $k$-EP.

As proposed in~\cite{Karisch98semidefinite}, {\color{black}one can eliminate  $Y  {\mathbf 1}_n = \frac{n}{k} {\mathbf 1}_n$ in \eqref{eq:relax_dnn}} and project the relaxation onto a smaller dimensional space, by exploiting the following result.
\begin{lemma}[\cite{Karisch98semidefinite}]  \label{lemmaParameter}
Let $V\in \mathbb{R}^{n\times (n-1)}$ such that $V^\top {\mathbf 1}_n =\mathbf{0}$ and $\mbox{rank}(V)=n-1$. Then
\begin{align*}
\bigg \{ Y \in \sn: \diag(Y) = {\mathbf 1}_n,  Y {\mathbf 1}_n=\frac{n}{k} {\mathbf 1}_n \bigg \} = 
 \left \{ \frac{1}{k}{\mathbf J}_n +VRV^\top : R\in {\mathcal S}^{n-1},~ {\diag}(VRV^\top) = \frac{k-1}{k} {\mathbf 1}_n \right \}.
\end{align*}
\end{lemma} 
The matrix $V$ in Lemma~\ref{lemmaParameter} can be any matrix which columns form a basis for ${\mathbf 1}_n^\perp$, e.g.,
\begin{align}\label{Vp}
V = \begin{pmatrix} {\mathbf I}_{n-1} \\[1ex] -{\mathbf 1}^\top_{n-1}\end{pmatrix}
\end{align}
We use the result of  Lemma~\ref{lemmaParameter}
and replace $Y$ by $\frac{1}{k}{\mathbf J}_n+VRV^\top$  in $(DNN_{EP})$, which leads to the following equivalent relaxation:
\begin{equation}\label{sdp-p}
   \begin{aligned}
        \min_R~ & ~~\langle L_{EP},V R V^\top\rangle \\ 
        \textrm{s.t.}~& \textrm{diag}(V R V^\top) = \frac{k-1}{k}{\mathbf 1}_n,\\
            & V R V^\top \geq -\frac{1}{k} {\mathbf J}_n, \quad  R \succeq \mathbf{0}, 
    \end{aligned} 
\end{equation}
were $R\in {\mathcal S}^{n-1}_+$. Here, we exploit $\langle L,  {\mathbf J}_n \rangle=0$ to  rewrite the objective, and define
\begin{align}\label{def:Leq}
L_{EP}:=\frac{1}{2}L.    
\end{align}
It is not difficult to verify that the matrix
$$\hat{R} = \frac{n(k-1)}{k(n-1)}\mathbf{I}_{n-1} - \frac{(k-1)}{k(n-1)}{\mathbf J}_{n-1} $$
is feasible for \eqref{sdp-p}, see also \cite{Karisch98semidefinite}. The matrix $\hat{R}$ 
has two distinct eigenvalues, namely  $\frac{n(k-1)}{k(n-1)}$
with multiplicity $n-2$ and $\frac{(k-1)}{k(n-1)}$ with multiplicity one. This implies that $\hat{R} \succ 0$. Also,
$$
\frac{1}{k} {\mathbf J}_n + V \hat{R} V^\top = \frac{n(k-1)}{k(n-1)} \mathbf{I}_n + \frac{(n-k)}{k(n-1)} {\mathbf J}_n> \mathbf{0},
$$
and thus $V\hat{R}V^\top > -\frac{1}{k}{\mathbf J}_n$. This shows that $\hat{R}$ is a Slater feasible point of \eqref{sdp-p}.

For future reference, we define the following sets:
\begin{align}
  \mathcal{R}_{EP} & :=\left \{R \in {\mathcal S}^{n-1} : R\succeq \mathbf{0} \right \},  \label{setR}\\
  \mathcal{X}_{EP} & := \bigg \{ X \in  {\mathcal S}^n :~\textrm{diag}({X}) = \frac{k-1}{k} {\mathbf 1}_n, ~ -\frac{1}{k} {\mathbf J}_n \leq  {X} \leq \frac{k-1}{k} {\mathbf J}_n ~  \bigg \}.\label{setX}  
\end{align}
Now, we rewrite the DNN relaxation \eqref{sdp-p}  as follows: 
\begin{align}\label{sdp-p2}
\min \left \{ \langle L_{EP},{X}\rangle: ~{X} = V R V^\top,~R \in \mathcal{R}_{EP},~{X} \in \mathcal{X}_{EP} \right \}.
\end{align}
 Note that $\mathcal{X}_{EP}$ also contains constraints that are redundant for \eqref{sdp-p}. \textcolor{black}{These constraints speed up the convergence of our algorithm, as explained in Section~\ref{sect:ADMMbasic}. In the same section, it becomes clear that the inclusion of redundant constraints should not complicate the structure of $\mathcal{X}_{EP}$ too much. Whether or not to include a redundant constraint, is determined by an empirical trade-off between these measures.}

\subsection{DNN relaxations for the bisection problem} \label{sect:bisection}

For the graph bisection problem there exist both vector and matrix lifting SDP relaxations. The matrix lifting relaxations derived in \cite{KarischRendlClausen,sotirov2014gpp} are equivalent and have matrix variables of order $n$.
A vector lifting SDP relaxation for the GBP is derived by  Wolkowicz and Zhao \cite{wokowicz-zhao1999} and has a matrix variable of order $2n+1$. The DNN relaxation from~\cite{wokowicz-zhao1999} dominates the basic matrix lifting DNN relaxations, i.e., DNN relaxations without additional cutting planes, see~\cite{sotirov2014gpp} for a proof.
In~\cite{Sotirov18} a matrix lifting DNN relaxation with additional cutting planes is derived for the GBP that is equivalent to the DNN relaxation from~\cite{wokowicz-zhao1999}.
Although the  relaxation from \cite{Sotirov18} has a matrix variable of order $n$, we  work with the vector lifting  DNN  relaxation because it has a more appropriate structure for our ADMM approach.

In this section we present the vector lifting DNN relaxation from~\cite{wokowicz-zhao1999} and show how to obtain its facially reduced equivalent version by using properties of the bisection polytope. As a byproduct, we also study properties of the feasible set of the DNN relaxation, see Theorem~\ref{Thm:nullspaceZW}.  

Let ${\mathbf m}=(m_1,m_2)^\top$ such that  $m_1+ m_2=n$ be given. 
To derive a vector lifting SDP relaxation for the GBP we linearize the objective from  $(GPP)$ by lifting variables into ${\mathcal S}^{2n+1}$.
In particular, let $P\in \{0,1\}^{n\times 2}$ be a partition matrix and $x=\vect(P)$. We use \eqref{kron:property1} and \eqref{kron:property2}  to rewrite the objective as follows:
$$
\trace (LPP^\top) = \vect (P)^\top (\mathbf{I}_2 \otimes L) \vect(P) = x^\top (\mathbf{I}_2 \otimes L)x = \langle \mathbf{I}_2 \otimes L, xx^\top \rangle.
$$
Now, we replace $xx^\top$ by a matrix variable $\hat{X} \in {\mathcal S}^{2n}$.
The constraint $\hat{X}=xx^\top$ can be weakened to $\hat{X}-xx^\top \succeq \mathbf{0}$, which is  equivalent to $X := \begin{pmatrix} 1 & x^\top \\
x & \hat{X}
\end{pmatrix} \succeq \mathbf{0}
$
by the well-known Schur Complement lemma.

In the sequel, we use the following block notation for matrices in ${\mathcal S}^{2n+1}$:
$$
{X} =
 \begin{pmatrix}
1 & (x^1)^\top & (x^2)^\top \\
x^1 & X^{11} & X^{12} \\
x^2 & X^{21} & X^{22} \\
 \end{pmatrix},
$$
where $x^1$ (resp., $x^2$) corresponds to the first (resp., second) column in $P$, and $X^{ij}$ corresponds to  
$x^i (x^j)^\top$ for $i,j=1,2$.

Now, from ${\mathbf 1}_n^\top x^i=m_i$ ($i=1,2$) it follows that $\trace (X^{ii})=m_i$,  $\trace ({\mathbf J}_n X^{ii})=m_i^2$   and $\trace({\mathbf J}_n(X^{12}+X^{21})) = 2m_{1}m_{2}$.
From $x^1 \circ x^2 = \mathbf{0}$ it follows that $\textrm{diag}(X^{12})=\mathbf{0}$.

The above derivation results in the  following vector lifting SDP relaxation for the GBP~\cite{wokowicz-zhao1999}:
\begin{align}\label{eq:ZW}
({SDP}_{BP}) & \qquad
\begin{aligned}
\min_X ~~&~  \frac{1}{2} \langle  L, X^{11}+X^{22} \rangle \\[1ex]
\textrm{s.t.}~~&~   \trace (X^{ii}) = m_{i}, ~ \trace~ ({\mathbf J}_n X^{ii}) = m_{i}^{2},  ~~i=1,2, \\[1ex]
&\textrm{diag}(X^{12})=\mathbf{0},~
\textrm{diag}(X^{21})=\mathbf{0},
~ \trace \left({\mathbf J}_n(X^{12}+X^{21})\right) = 2m_{1}m_{2}, \\[1ex]
& 
X= \begin{pmatrix}
1 & (x^1)^\top & (x^2)^\top \\
x^1 & X^{11} & X^{12} \\
x^2 & X^{21} & X^{22} \\
 \end{pmatrix} \succeq \mathbf{0},
~~x^i=\textrm{diag}(X^{ii}), ~~ i=1,2,
\end{aligned}
\intertext{where $X\in {\mathcal S}^{2n+1}$. 
By imposing nonnegativity constraints on the matrix variable in $({SDP}_{BP})$, we obtain the following DNN relaxation:} 
({DNN}_{BP}) & \qquad (SDP_{BP}) ~~\&~~ X \geq \mathbf{0}.\label{DNNBP}
\end{align}
 The relaxation $(DNN_{BP})$ can be further strengthened by additional cutting planes. We propose adding the boolean quadric polytope inequalities, see Section~\ref{subsect:bqp}. 

The zero pattern on off-diagonal blocks in \eqref{eq:ZW} can be  written using a linear operator $\mathcal{G}_{\mathcal J} (\cdot)$, known as the Gangster operator, see~\cite{wokowicz-zhao1999}. The operator ${\mathcal G}_{\mathcal J}\colon {\mathcal S}^{2n+1} \to {\mathcal S}^{2n+1}$ is defined as 
$$
\mathcal{G}_{\mathcal J} (X)=\left \{
\begin{array}{ll}
{X}_{ij}     & \mbox{ if } (i,j)\in {\mathcal J}, \\
0     & \mbox{ otherwise,}
\end{array}\right .
$$
where
\begin{align} \label{Def:J}
{\mathcal J} = \left \{ (i,j) ~:~ i=(p-1)n+q+1, ~j=(r-1)n+q+1, ~q\in [n],~p,r\in \{1,2\},~p\neq r \right \}. 
\end{align}
The  constraints $\diag(X^{12}) = \textrm{diag}(X^{21})=\mathbf{0}$  are given by $\mathcal{G}_{J}\left( X \right) =\mathbf{0}$.  

We now show how to project the  SDP relaxation \eqref{eq:ZW} onto a smaller dimensional space in order to obtain an equivalent strictly feasible relaxation by facial reduction. Although such reduction is performed for the general graph partitioning problem in \cite{wokowicz-zhao1999}, our approach differs  by relying on the polytope of all bisections. We first apply facial reduction to the relaxation $(SDP_{BP})$, after which we derive the facially reduced equivalent of $(DNN_{BP})$.

We start by deriving two properties that hold for all feasible solutions of $(SDP_{BP})$. 
\begin{theorem} \label{Thm:nullspaceZW}
Let $X = \begin{pmatrix}
1 & (x^1)^\top & (x^2)^\top \\
x^1 & X^{11} & X^{12} \\
x^2 & X^{21} & X^{22} \\
 \end{pmatrix}$ with $\hat{X} = \begin{pmatrix}
 X^{11} & X^{12} \\
X^{21} & X^{22} \\
 \end{pmatrix}$ and $x = \begin{pmatrix}
 x^1 \\ x^2
 \end{pmatrix}$ be feasible for $(SDP_{BP})$. Then,
\begin{enumerate}[label=(\roman*)]
\item $a_i^\top \left( \hat{X} - xx^\top \right)a_i = 0$ where $a_i = {\mathbf u}_i \otimes {\mathbf 1}_n$, ${\mathbf u}_i \in \mathbb{R}^2$, $i \in [2]$; 
\item $b_i^\top \left( \hat{X} - xx^\top \right)b_i = 0$ where $b_i = {\mathbf 1}_2 \otimes {\mathbf u}_i$, ${\mathbf u}_i \in \mathbb{R}^n$, $i \in [n] $. 
\end{enumerate}
\end{theorem}
\begin{proof}
$(i)$ Without loss of generality we take $i = 1$. Then $a_1 = {\mathbf u}_1\otimes {\mathbf 1}_n$,  which yields
\begin{align*}
a_1^\top \left( \hat{X} - xx^\top \right) a_1 & = {\mathbf 1}_n^\top X^{11} {\mathbf 1}_n - {\mathbf 1}_n^\top x^1(x^1)^\top {\mathbf 1}_n = \trace({\mathbf J}_n X^{11}) - \trace(X^{11})^2 = m_1^2 - m_1^2 = 0,
\end{align*}
using the constraints of \eqref{eq:ZW}. The proof for $i = 2$ is similar.

$(ii)$ First, we show that any feasible solution to \eqref{eq:ZW} satisfies $\diag(X^{11}) + \diag(X^{22}) = {\mathbf 1}_n$. For all $i \in [n]$ we define $v^i \in \mathbb{R}^{2n}$ as 
\begin{align*}
\left( v^i \right)_j := \begin{cases} -1 & \text{if $j = i$ or $j = n+i$,} \\
0 & \text{otherwise.}
\end{cases}
\end{align*}
From $\hat{X} - xx^\top \succeq {\mathbf 0}$, we have
\begin{align*}
\begin{pmatrix}
1 \\
v^i
\end{pmatrix}^\top \begin{pmatrix}
1 & x^\top \\
x &\hat{X}
\end{pmatrix}\begin{pmatrix}
1 \\
v^i
\end{pmatrix} \geq 0 \quad \Longleftrightarrow \quad 1 - X_{ii}^{11} - X_{ii}^{22} \geq 0 \quad \Longleftrightarrow \quad X_{ii}^{11} +X_{ii}^{22} \leq 1,
\end{align*}
where we used the fact that $\diag(\hat{X}) = x$. Since
\begin{align*}
n = m_1 + m_2 = \trace(X^{11}) + \trace(X^{22}) = \sum_{i = 1}^n \left(  X_{ii}^{11} +X_{ii}^{22} \right), 
\end{align*}
and the latter summation consists of $n$ elements bounded above by $1$, we must have $X_{ii}^{11} +X_{ii}^{22} = 1$ for all $i \in [n]$. 

Now, for $b_i = {\mathbf 1}_2 \otimes {\mathbf u}_i$, $i\in [n]$ we have
\begin{align*}
b_i^\top \left( \hat{X} - xx^\top \right) b_i & = X_{ii}^{11} + X_{ii}^{12} + X_{ii}^{21} + X_{ii}^{22} - \left( X_{ii}^{11} + X_{ii}^{22} \right)^2. 
\end{align*}
Since $\diag(X^{12}) = \diag(X^{21}) = \mathbf{0}$ and $\diag(X^{11}) + \diag(X^{22}) = {\mathbf 1}_n$, it follows that $b_i^\top \big( \hat{X} - xx^\top \big) b_i = 1 - 1^2 = 0$ for all $i\in[n]$.
\end{proof}
We can exploit the properties stated in Theorem~\ref{Thm:nullspaceZW} to identify vectors in the null space of all feasible solutions of $(SDP_{BP})$. In order to do so, we use the following result.  
\begin{lemma}[\cite{RendlSotirov}] \label{Lem:Eigenvector} Let $X \in \mathcal{S}^l$, $x \in \mathbb{R}^l$ and $a \in \mathbb{R}^l$ be such that $X - xx^\top \succeq \mathbf{0}$, $a^\top x = t$ for some $t \in \mathbb{R}$, and $a^\top \left( X - xx^\top \right) a = 0$. Then $[ -t, a^\top ]^\top$ is an eigenvector of $\begin{pmatrix}
1 & x^\top \\ x & X
\end{pmatrix}$ with respect to  eigenvalue 0. 
\end{lemma}
It follows from the constraints of \eqref{eq:ZW} that $a_i^\top x = m_i$ for $i \in [2]$ and $b_i^\top x = 1$ for $i \in [n]$, where $a_i$ and $b_i$ are defined as in Theorem \ref{Thm:nullspaceZW}. As a result, Theorem~\ref{Thm:nullspaceZW} and Lemma~\ref{Lem:Eigenvector} imply that 
\begin{align*}
\begin{pmatrix}
-m_i \\
{\mathbf u}_i \otimes {\mathbf 1}_n
\end{pmatrix}, \, i \in [2], \quad \text{and} \quad 
\begin{pmatrix}
-1 \\
{\mathbf 1}_2 \otimes {\mathbf u}_i
\end{pmatrix}, \, i  \in [n],
\end{align*}
are eigenvectors of $\begin{pmatrix}
1 & x^\top \\ x & \hat{X}
\end{pmatrix}$ with respect to eigenvalue 0. Now, let us define the matrix $T\in \mathbb{R}^{(n+2) \times (2n + 1)}$ as follows: 
\begin{align*}
T := \begin{pmatrix}
- {\mathbf m} & \mathbf{I}_2 \otimes {\mathbf 1}_n^\top  \\
- {\mathbf 1}_n & {\mathbf 1}_2^\top \otimes \mathbf{I}_n
\end{pmatrix}. 
\end{align*}
Moreover, let $\mathcal{V} = \Nul(T)$. Any $a \in \mathcal{V}$ defines an element $aa^\top$ exposing the feasible set of $(SDP_{BP})$. It follows from the facial geometry of the cone of positive semidefinite matrices that the feasible set of $(SDP_{BP})$ is contained in 
\begin{align*}
S_\mathcal{V} := \left\{
X \in \mathcal{S}^{2n+1}_+ \, : \, \, \Col(X) \subseteq \mathcal{V} \right\},
\end{align*}
which is a face of $\mathcal{S}^{2n+1}_+$. It remains to prove that this is actually the minimal face of $\mathcal{S}_+^{2n+1}$ containing the feasible set of $(SDP_{BP})$. For that purpose, we consider the underlying bisection polytope $\Conv(F_n^2({\mathbf m}))$, see \eqref{Def:GPPpolytope}. It follows from Theorem~\ref{Thm:DimConvF} that $\dim(\Conv(F_n^2({\mathbf m}))) = n -1$. Besides, observe that $T$ is constructed as the constraint matrix defining $\Conv(F_n^2({\mathbf m}))$ augmented with an additional column. Since this additional column does not increase its rank, we have $\rank(T) = n+1$, which implies that $\dim(\mathcal{V}) = n$. Let $V \in \mathbb{R}^{(2n + 1) \times n}$ be a matrix whose columns form a basis for $\mathcal{V}$. Then the face $S_\mathcal{V}$ can be equivalently written as: 
\begin{align}
S_\mathcal{V} = V \mathcal{S}_+^{n} V^\top. 
\end{align}
To show that $S_\mathcal{V}$ is the minimal face containing the feasible set of $(SDP_{BP})$ we apply a result by Tun\c{c}el \cite{Tuncel}. 
\begin{theorem}[\cite{Tuncel}] \label{Thm:Tuncel}
Given $F \subseteq \mathbb{R}^{l}$, let $\mathcal{F} := \left\{ \begin{pmatrix}
1 & x^\top \\ x & \hat{X}
\end{pmatrix} \in \mathcal{S}_+^{l+1} \, : \right. \left. \, \, \mathcal{A} \left( \begin{pmatrix}
1 & x^\top \\ x & \hat{X}
\end{pmatrix} \right) = \mathbf{0} \right\}$, with $\mathcal{A}\colon \mathcal{S}^{l+1} \to \mathbb{R}^p$ a linear transformation, be a relaxation of the lifted polyhedron
\begin{align*}
\Conv \left\{ \begin{pmatrix}
1 \\ x
\end{pmatrix}\begin{pmatrix}
1 \\ x
\end{pmatrix}^\top \, : \, \, x \in F \right\}.
\end{align*}
Suppose that $\mathcal{F} \subseteq V \mathcal{S}^d_+ V^\top$ for some full rank matrix $V \in \mathbb{R}^{(l+1)\times d}$. If $\dim(\Conv(F)) = d-1$, then there exists some $R \succ \mathbf{0}$ such that $VRV^\top \in \mathcal{F}$.  
\end{theorem}
Based on Theorem~\ref{Thm:Tuncel}, we can now show the minimality of the face $S_\mathcal{V}$ for both $(SDP_{BP})$ and $(DNN_{BP})$. 
\begin{theorem} \label{Thm:minimalface}
The set $S_\mathcal{V}$ is the minimal face of $\mathcal{S}^{2n+1}_+$ containing the feasible set of $(SDP_{BP})$. If $m_1,m_2 \geq 2$, then $S_\mathcal{V}$ is also the minimal face of $\mathcal{S}^{2n+1}_+$ containing the feasible set of $(DNN_{BP})$. 
\end{theorem}
\begin{proof}
The feasible region of $(SDP_{BP})$ can be written in the form of $\mathcal{F}$  in the statement of Theorem~\ref{Thm:Tuncel}. To show minimality for $(SDP_{BP})$, it suffices to show that there exists a matrix $R \in \mathcal{S}^n_+$, $R \succ \mathbf{0}$ such that $VRV^\top$ is feasible for $(SDP_{BP})$. Since $\dim(\Conv(F_n^2({\mathbf m}))) = n-1$, it immediately follows from Theorem~\ref{Thm:Tuncel} that such matrix, say $R_1$, exists.

To prove the second statement, it suffices to show that there exists an $R \in \mathcal{S}^n_+$ such that $R \succ \mathbf{0}$ and $(VRV^\top)_{ij} > 0$ for all $(i,j) \in {\mathcal J}_C$, where ${\mathcal J}_C = ([2n+1] \times [2n+1])\setminus {\mathcal J}$, see \eqref{Def:J}. Since $m_1, m_2 \geq 2$, it follows that for any $(i,j) \in {\mathcal J}_C$ there exists a bisection $x^{ij}$ such that 
\begin{align*} \left(
    \begin{pmatrix}
    1 \\ x^{ij}
    \end{pmatrix}\begin{pmatrix}
    1 \\ x^{ij}
    \end{pmatrix}^\top \right)_{ij} > 0. 
\end{align*}
Let $R^{ij} \in \mathcal{S}^n_+$ denote the matrix such that $VR^{ij}V^\top = \begin{pmatrix}
1 \\ x^{ij}
\end{pmatrix}\begin{pmatrix}
1 \\ x^{ij}
\end{pmatrix}^\top$, which consists by construction of $S_\mathcal{V}$. Now, let $R_2$ be any positive convex combination of the elements in $\{R^{ij} \, : \, \, (i,j) \in {\mathcal J}_C\}$. By construction, $R_2 \succeq \mathbf{0}$, while $(VR_2V^\top)_{ij} > 0$ for all $(i,j) \in {\mathcal J}_C$. Finally, any positive convex combination of $R_1$ and $R_2$ provides a matrix $R$ with the desired properties.
\end{proof}
The result of Theorem~\ref{Thm:minimalface} can be exploited to derive strictly feasible equivalent versions of $(SDP_{BP})$ and $(DNN_{BP})$. We focus here only on the DNN relaxation $(DNN_{BP})$. 
Theorem~\ref{Thm:minimalface} allows us to replace $X$ by $VRV^\top$ in $(DNN_{BP})$.
One can take the following matrix for $V$:
\begin{align} \label{V2n}
V = \begin{pmatrix}
1 & \mathbf{0} \\
\frac{1}{n}{\mathbf m} \otimes \mathbf{1}_n & V_2\otimes V_n
\end{pmatrix}, \quad \text{where } V_p = \begin{pmatrix}
{\mathbf I}_{p-1} \\ -{\mathbf 1}_{p-1}^\top 
\end{pmatrix}   \text{ for } p = 2, n.
\end{align}
Because of the structure of $V$, most of the constraints in $(DNN_{BP})$ become redundant. One can easily verify that the resulting relaxation in lower dimensional space is as follows, see e.g.,~\cite{wokowicz-zhao1999}: 
\begin{equation}\label{ZWslater}
 \quad \begin{aligned}
 \min_R ~& ~~\trace ( L_{BP} {V}R{V}^\top) \\
\textrm{s.t.} ~&~~ {\mathcal G}_{\mathcal J}({V}R{V}^\top) = \mathbf{0}, \\
                      &~~ ({V}R{V}^\top)_{1,1}=1,\\
                      &~~ {V}R{V}^\top \geq \mathbf{0},~ R\in {\mathcal S}_+^{n},
\end{aligned}
\end{equation}
where
\begin{align}\label{def:Lbp}
L_{BP} := \frac{1}{2} \begin{pmatrix}
0 & \mathbf{0}^\top\\
\mathbf{0} & \mathbf{I}_2 \otimes L
\end{pmatrix},
\end{align}
and $L$ is the weighted Laplacian matrix of $G$. Let us now define the following sets: 
\begin{align}
\mathcal{R}_{BP} & :=\left \{R \in {\mathcal S}^{n} : R\succeq \mathbf{0} \right \}, \label{setRGB} \\
\mathcal{X}_{BP} &:= \left \{  
X \in  {\mathcal S}^{2n+1} : ~~ \begin{aligned} & {\mathcal G}_{\mathcal J}({X}) = \mathbf{0}, ~{X}_{1,1}=1,~\trace(X^{ii})=m_i,~ i\in [2], \\
& \diag(X^{11})+\diag(X^{22})={\mathbf 1}_n,~
X {\mathbf u}_1 =\diag(X),\\
& \mathbf{0} \leq {X} \leq {\mathbf J} 
\end{aligned}  \right\}. \label{setXGB}
\end{align}
Now, we are ready to rewrite the facially reduced DNN relaxation \eqref{ZWslater} as follows:
\begin{equation} \label{sdp-bp2}
    \min \left \{\langle L_{BP},X\rangle: ~X = V R V^\top,~R \in \mathcal{R}_{BP},~X \in \mathcal{X}_{BP} \right \}.
\end{equation}
 Note that $\mathcal{X}_{BP}$ also contains constraints that are redundant for \eqref{ZWslater}.

\section{A cutting plane augmented Lagrangian algorithm} \label{sect:admm}

{\color{black}
SDP has proven effective for modeling optimization problems and providing strong bounds. It is well-known that SDP solvers based on interior point methods might  have considerable memory demands already  for medium-scale problems. 
Recently,  promising alternatives for solving large-scale SDP relaxations have been investigated.
We refer the interested reader to \cite{BurerVandenbussche,MaiEtAl2, PovhEtAl,SunTohetAl,WenEtAl,YurtseverEtAl} for algorithms based on alternating direction augmented Lagrangian methods for solving  SDPs.
For efficient approaches \textcolor{black}{to} solving DNN relaxations, see also e.g.,~\cite{HuSotirov,{HaoSotWolk},Li2021ASC,oliveira2018admm,wiegele-zhao2022,zhao2022}. 
To the best of our knowledge only \cite{deMeijerSotirov} incorporates an augmented Lagrangian method into a cutting-plane framework.
The authors of \cite{deMeijerSotirov} consider only one type of cutting planes.
Here, we incorporate various types of cutting planes into one framework and use a more efficient version of the ADMM than the one used in \cite{deMeijerSotirov}.
}

In Section~\ref{sect:ADMMbasic} we describe variants of the ADMM that are used within our cutting-plane algorithm. Section~\ref{sect:Dykstra} presents Dykstra's cyclic projection algorithm that is used for projections onto polyhedra induced by the violated cuts. Section \ref{sect:CPADMM} presents our cutting plane augmented Lagrangian algorithm.

\subsection{The Alternating Direction Method of Multipliers } \label{sect:ADMMbasic}

The ADMM is a first-order method from the 1970s that is developed for solving convex optimization problems. This method decomposes an optimization problem into several subproblems that are easier to solve than the original problem.
There exist several variants of the ADMM for solving SDPs.
We consider here a variant of the ADMM that resembles variants from  \cite{HuSotirov,oliveira2018admm}, 
where we additionally consider an adaptive stepsize term proposed by Lorenz and Tran-Dinh~\cite{lorenz2018non} when solving the $k$-EP.

In order to describe the ADMM scheme for solving SDP relaxations for both problems,  the $k$-equipartition problem \eqref{sdp-p2} and the graph bisection problem \eqref{sdp-bp2}, we introduce the following unified notation:
For the $k$-EP,  define $\bar{L}:=L_{EP}$, $\mathcal{R}:=\mathcal{R}_{EP}$ and $\mathcal{X}:=\mathcal{X}_{EP}$
(see resp.,~\eqref{def:Leq}, \eqref{setR}, \eqref{setX}),
and for the GBP  define
$\bar{L}:=L_{BP}$, $\mathcal{R}:=\mathcal{R}_{BP}$ and $\mathcal{X}:=\mathcal{X}_{BP}$ (see resp.,~\eqref{def:Lbp}, \eqref{setRGB}, \eqref{setXGB}).

Let $Z$ denote the Lagrange multiplier for the  constraint $X=VRV^\top$. Then, the augmented Lagrangian function of \eqref{sdp-p2} and \eqref{sdp-bp2}  w.r.t.~the constraint $X=VRV^\top$ for a  penalty parameter $\sigma$
is as follows:
\begin{equation} \label{eq:Lagrangianfunction}
\mathcal{L}_{\sigma}({X},R,Z )=  \langle \bar{L},{X}\rangle + \langle Z,{X}-{V}R{V}^\top\rangle + 
\frac{\sigma}{2} \|{X}-VRV^\top\|^2_F.
\end{equation}
In each iteration, the ADMM minimizes $\mathcal{L}_{\sigma}({X},R,Z )$ subject to $X\in \mathcal{X}$ and $R\in \mathcal{R}$ and  updates $Z$ via a stepsize update.
The ADMM update scheme requires a matrix $V$ that has orthonormal columns that can be obtained by applying a QR-decomposition to \eqref{Vp} for the $k$-EP and to \eqref{V2n} for the GBP.

Let $(R^p,X^p,Z^p)$ denote the $p$-th iterate of the ADMM. The next iterate $(R^{p+1},X^{p+1},Z^{p+1})$  is obtained as follows:
\begin{subequations}\label{ADMMall1}
\begin{align}
    R^{p+1} =& \argmin_{R\in \mathcal{R}}\mathcal{L}_{\sigma^p}(R,{X}^p,Z^p),\label{R_sub}\\ 
    X^{p+1} =& \argmin_{X\in \mathcal{X}}\mathcal{L}_{\sigma^p}(R^{p+1},{X},Z^p), \label{X_sub}\\
    Z^{p+1} =& Z^p +  \gamma \cdot \sigma^{p} \cdot (X^{p+1}-VR^{p+1}V^\top), \label{Z_sub}
    \end{align}
\end{subequations}
where $\gamma \in \left(0, \frac{1 + \sqrt{5}}{2} \right )$ is a parameter for updating the dual multiplier $Z^p$, see e.g., \cite{WenEtAl}.

There exist different ways for dealing with the stepsize term $\gamma \cdot \sigma^p$. 
One possibility is to keep $\sigma^p$ and $\gamma$  fixed during the algorithm. In this approach, $\sigma^p$ depends on the problem data and  $\gamma$ has a value larger than one.
This is known in the literature as the ADMM with larger stepsize, as originally proposed by~\cite{FortinGlowinski}. An alternative is the ADMM with adaptive stepsize term as introduced in~\cite{lorenz2018non}. In that case $\gamma = 1$ and the parameter $\sigma^{p}$
 is updated as follows: 
 \begin{equation}\label{updateParam}
   \sigma^{p+1} :=  (1-\omega^{p+1})\sigma^p + \omega^{p+1} \mathcal{P}_{[\sigma_{\min},\sigma_{\max}] }\frac{\|Z^{p+1}\|_F}{\|X^{p+1}\|_F},
\end{equation}
where $\omega^{p+1}:= 2^{-p/100}$ is the weight, $\sigma_{\min}$ and $\sigma_{\max}$ are the box bounds for  $\sigma^p$, and
$\mathcal{P}_{[\sigma_{\min},\sigma_{\max}]}$ is the projection onto  $[\sigma_{\min},\sigma_{\max}]$.

Recall that we added redundant constraints for the SDP relaxations \eqref{sdp-p2} and \eqref{sdp-bp2} to the set $\mathcal{X}$. Those constraints are, though, not redundant in the subproblem \eqref{X_sub}. They are included to \textcolor{black}{speed up} the convergence of the ADMM algorithm in practice, see e.g.,~\cite{HuSotirov,oliveira2018admm,deMeijerSotirov}.

One can solve the $R$-subproblem~\eqref{R_sub} as follows:
\begin{align*}
  R^{p+1} =   &\argmin_{R\in \mathcal{R}}\mathcal{L}_{\sigma^p}(R,{X}^p,Z^p) = \argmin_{R\in \mathcal{R}} ~\langle Z^p,-VRV^\top\rangle + \frac{\sigma^p}{2} \left \|{X}^p-VRV^\top \right \|^2_F,\\
   =& \argmin_{R\in \mathcal{R}} ~\left \| V^\top \left( {X}^p+ \frac{1}{\sigma^p} Z^p \right )V- R \right \|_F^2 = \mathcal{P}_{\succeq \mathbf{0}} \left ( V^\top \left( {X}^p + \frac{1}{\sigma^p} Z^p \right )V \right ),
\end{align*}
where $\mathcal{P}_{\succeq \mathbf{0}} (\cdot)$ denotes the orthogonal projection onto the cone of positive semidefinite matrices.

The $X$-subproblem \eqref{X_sub} can be solved as follows:
\begin{align*}
  X^{p+1} =   &\argmin_{{X}\in \mathcal{X}}\mathcal{L}_{\sigma^p}(R^{p+1},{X},Z^p)= \argmin_{{X}\in \mathcal{X}} ~\langle \bar{L}+ Z^p,{X}\rangle  + \frac{\sigma^p}{2} \left \|{X}-VR^{p+1}V^\top \right \|^2_F,\\
    =& \argmin_{{X}\in \mathcal{X}} \left \|{X}- \left (VR^{p+1}V^\top- \frac{1}{\sigma^p}\left (\bar{L}+Z^p \right ) \right )\right \|^2_F = \mathcal{P}_{\mathcal{X}}\left (VR^{p+1}V^\top- \frac{1}{\sigma^p} \left (\bar{L}+Z^p \right) \right ), 
\end{align*} 
where $ \mathcal{P}_{\mathcal{X}}(\cdot)$ denotes the orthogonal projection onto the polyhedral set $\mathcal{X}$. In Appendix~\ref{App:projectorX} we show how this projection can be performed explicitly.

\medskip
{\color{black}The performance of the ADMM greatly depends on the stepsize term.}
Our preliminary tests show that for the $k$-EP the updating scheme \eqref{ADMMall1}--\eqref{updateParam} with adaptive stepsize term outperforms the ADMM with larger stepsize.  {\color{black}That is, our adaptive ADMM performs better than the ADMM variants from \cite{HuSotirov,{HaoSotWolk},oliveira2018admm}. Moreover, preliminary results show that it outperforms  the algorithm from~\cite{Li2021ASC}.} 
For the GBP, however, our preliminary tests show that it is more beneficial to keep $\sigma^p$ fixed and use  larger  $\gamma$. 
{\color{black}The resulting version of the ADMM resembles versions from~\cite{HuSotirov,{HaoSotWolk},oliveira2018admm}.}
Consequently, we initialize the ADMM algorithm \eqref{ADMMall1} for the $k$-EP by \begin{align}\label{ADMM:initialEP}
R^0 = \mathbf{0}, \qquad {X}^0=\frac{k-1}{k}\mathbf{I}_n, \qquad Z^0=\mathbf{0},\qquad \sigma^0= \left \lceil \frac{n}{k}\right  \rceil,   \qquad \gamma = 1,
\end{align}
and for the GBP we set 
\begin{align}\label{ADMM:initialGBP}
R^0  = \mathbf{0}, \qquad 
X^0  = {\mathbf u}_1 {\mathbf u}_1^\top, 
 \qquad Z^0=\mathbf{0}, \qquad {\sigma^0 =  \left \lceil \left (\frac{2n}{m_1} \right )^2 \right \rceil, \qquad \gamma = 1.608}.   
\end{align}

\subsection{Clustered Dykstra's projection algorithm} \label{sect:Dykstra}
Both  DNN problems, \eqref{sdp-p2} and \eqref{sdp-bp2}, can be strengthened by adding valid cutting planes, see Section~\ref{sect:cuts}. Since these cutting planes are polyhedral, it is natural to include additional cuts to the set $\mathcal{X}$. This addition will, however, spoil the easy structure of $\mathcal{X}$. As a result, finding the explicit projection onto this new polyhedral set becomes a difficult task, even after the addition of a single cut. In \cite{deMeijerSotirov} this issue is resolved by splitting the polyhedral set into subsets and using iterative projections based on Dykstra's algorithm~\cite{BoyleDykstra, Dykstra}. This algorithm finds the projection onto the intersection of a finite number of polyhedral sets, assuming that the projection onto each of the separate sets is known. Although there exist many algorithms for finding such projection in the literature, the recent study of~\cite{BauschkeKoch} shows superior behaviour of Dykstra's cyclic projection algorithm. 

In this section we briefly present Dykstra's algorithm and show how to implement it 
efficiently by clustering non-overlapping cuts.
Similar to the previous section, we present a generic version of the algorithm that can be applied to both the $k$-EP and the GBP.

\medskip 
Let us assume that $\mathcal{T}$ is an index set of cutting planes on the primal variable $X$ in the ADMM scheme, i.e., every $t \in \mathcal{T}$ corresponds to a single cut. Also, for each $t \in \mathcal{T}$ let $\mathcal{H}_t$ be a polyhedron that is induced by the cut $t$. \textcolor{black}{One can think of $\mathcal{H}_t$ as the halfspace induced by the cut $t$, where additional constraints are added as long as the projection onto $\mathcal{H}_t$ remains efficient.} 
In Section~\ref{sect:cuts} we show how the set $\mathcal{T}$ and the polyhedra $\mathcal{H}_t$ look like for cuts related to both the $k$-EP and the GBP, and present the projectors onto the sets $\mathcal{H}_t$. 

When adding the cuts in $\mathcal{T}$ to the relaxation, the polyhedral set $\mathcal{X}$ has to be replaced by
\begin{align}\label{XTset}
    \mathcal{X}_\mathcal{T} := \mathcal{X} \cap \left( \bigcap_{t \in \mathcal{T}} \mathcal{H}_t \right). 
\end{align}
The $X$-subproblem of the ADMM scheme \eqref{X_sub} asks for the projection onto $\mathcal{X}_\mathcal{T}$. That is, for a given matrix $M$, one wants to solve the following best approximation problem: 
\begin{align} \label{ProjectionProblem}
\min_{\hat{M}}  ~& \Vert \hat{M} - M \Vert_F^2 \quad \text{ s.t. } \quad \hat{M} \in  \mathcal{X}_\mathcal{T}. 
\end{align}
Since the structure of $\mathcal{X}_\mathcal{T}$ is too complex to perform the projection in one step, the idea behind Dykstra's algorithm is to use iterative projections. Let $\mathcal{P}_{\mathcal{H}_t}( \cdot )$ denote the projection onto $\mathcal{H}_t$ for each $t \in \mathcal{T}$. Also, we assume that $\mathcal{P}_\mathcal{X}( \cdot )$ is known.

Dykstra's algorithm starts by initializing the so-called normal matrices $N^0_\mathcal{X} = \mathbf{0}$ and $N^0_{t} = \mathbf{0}$ for all   $t\in \mathcal{T}$. These normal matrices have the same size as the primal variable $X$ in the ADMM scheme. Moreover, we initialize   $X^0 = M$. 
The algorithm iterates for $q \geq 1$ as follows:
\begin{align} \label{AlgCyclicDykstra} \tag{CycDyk}
\begin{aligned}
\begin{aligned}
X^q & := \mathcal{P}_\mathcal{X} \left( X^{q-1} + N_\mathcal{X}^{q-1} \right) \\
N_{\mathcal{X}}^q & := X^{q-1} + N_\mathcal{X}^{q-1} - X^q \end{aligned}  \,\, \quad \, & \\
\left.
\begin{aligned}
L_{t} & := X^q + N_{t}^{q-1} \\
X^q & := \mathcal{P}_{\mathcal{H}_t} \left( L_{t} \right) \\
N^q_{t} & := L_{t} -   X^q
\end{aligned} \quad \right\} & \quad \text{for all } t \in \mathcal{T}
\end{aligned}
\end{align}
Since the polyhedra are considered in a cyclic order, the iterative scheme (\ref{AlgCyclicDykstra}) is also known in the literature as Dykstra's cyclic projection algorithm.
The sequence $\{X^q\}_{q\geq 1}$ strongly converges to the solution of the best approximation problem (\ref{ProjectionProblem}), see e.g.,~\cite{BoyleDykstra,Bichot,GaffkeMathar}.

We perform several actions to implement the algorithm \eqref{AlgCyclicDykstra} as \textcolor{black}{efficiently} as possible. First, we can reduce the number of iterations needed to converge by adding some of the constraints of $\mathcal{X}$ also to the sets $\mathcal{H}_t$. This brings the sets $\mathcal{H}_t$ closer to the intersection $\mathcal{X}_\mathcal{T}$, leading to faster convergence. A restriction on this addition is that we should still be able to find the explicit projection onto $\mathcal{H}_t$. In Section~\ref{subsect:bqp} we show how some of the constraints from  the DNN relaxation of the bisection problem are added to the polyhedra $\mathcal{H}_t$, while keeping the structure of the polyhedra sufficiently simple.

Second, as observed in \cite{deMeijerSotirov}, it is possible to partly parallelize the algorithm \eqref{AlgCyclicDykstra}. The cuts in $\mathcal{T}$ are often very sparse. This implies that the projection onto $\mathcal{H}_t$ only involves a small number of entries, while the other entries are kept fixed. This property can be exploited by clustering non-overlapping cuts. Suppose two cuts $t_1, t_2 \in \mathcal{T}$ are such that no entry in the matrix variable of the $X$-subproblem is perturbed  by both $\mathcal{P}_{\mathcal{H}_{t_1}} ( \cdot)$ and $\mathcal{P}_{\mathcal{H}_{t_2}} ( \cdot )$. Then, we can project onto both cuts simultaneously. This idea can be generalized by creating clusters of non-overlapping cuts. 
Suppose we cluster the set $\mathcal{T}$ into $r$ clusters $C_i$, $i\in [r]$ such that $C_1\cup \ldots \cup C_r = \mathcal{T}$, $C_i\cap C_j = \emptyset$ for $i \neq j$, $i,j\in [r]$, and all cuts in $C_i$, $i \in [r]$ are non-overlapping. 
Then, an iterate of  (\ref{AlgCyclicDykstra}) 
is performed in $r+1$ consecutive steps, instead of $|\mathcal{T}|+1$.

To cluster the cuts, we proceed as follows.
We denote by $H$ an undirected graph in which each vertex  represents a cutting plane indexed by an element from $\mathcal{T}$. 
Two vertices in $H$ are connected by an edge if and only if two cuts are overlapping. 
Clustering $\mathcal{T}$ into non-overlapping sets corresponds to clustering vertices of $H$ into independent sets. Therefore, clustering $\mathcal{T}$ into the smallest number of non-overlapping sets reduces to finding a minimum coloring in $H$.
Since the graph coloring problem is NP-hard, we use an efficient heuristic algorithm  from~\cite{GalinierHao} to find a near-optimal coloring.

\subsection{The cutting-plane ADMM } \label{sect:CPADMM}
In this section we put all elements of our cutting-plane algorithm together. In particular, we combine the ADMM from  Section~\ref{sect:ADMMbasic} and the clustered implementation of Dykstra's projection algorithm from Section~\ref{sect:Dykstra} into a cutting-plane ADMM-based algorithm. We refer to this algorithm as the cutting-plane ADMM. Algorithm~\ref{Alg:CP_ADMM} provides a  pseudo-code of our algorithm. 
{\color{black}
Since the CP--ADMM algorithm
solves  a two-block separable convex problem it is guaranteed to converge, see e.g.,~\cite{Boyd2011DistributedOA,WenEtAl}.} The stopping criteria and input parameters are specified in Section~\ref{sec:stopp} and Section~\ref{sect:numerics}, respectively.

The CP--ADMM algorithm is designed to  solve DNN relaxations for the GPP with additional cutting planes. In particular, 
Algorithm~\ref{Alg:CP_ADMM}  can solve the DNN relaxation for the $k$-EP, see~\eqref{sdp-p2}, that is strengthened by the triangle inequalities  \eqref{ineq:trianglehat}  and independent set inequalities~\eqref{ineq:cliquehat}. Similarly, Algorithm~\ref{Alg:CP_ADMM} also solves the DNN relaxation for the GBP, see~\eqref{sdp-bp2}, with additional BQP inequalities~\eqref{BQP3}.

Let us outline the main steps of the CP--ADMM. Initially, the set $\mathcal{T}$ is empty and the algorithm solves the basic DNN relaxation, i.e., the DNN relaxation without additional cutting planes, using the ADMM as described in Section~\ref{sect:ADMMbasic}. After one of the stopping criteria from the inner while-loop is satisfied, see Section~\ref{sec:stopp}, a valid lower bound is computed based on the current approximate solution, see Section~\ref{sect:lowerbound}. Then, the algorithm identifies violated cuts and adds the $numCuts$ most violated ones to $\mathcal T$. To increase performance, the cuts induced by tuples in $\mathcal T$ are clustered by using a heuristic for the graph coloring problem from~\cite{GalinierHao}. The procedure is repeated, where the projection onto ${\mathcal X}_{\mathcal T}$, see \eqref{XTset}, is performed by the semi-parallelized version of Dykstra’s projection algorithm, see Section~\ref{sect:Dykstra}. The outer while-loop stops whenever one of the global stopping criteria is met.

The CP--ADMM can be extended to solve various DNN relaxations with a large number of additional cutting planes. We remark that computing such strong  bounds was not possible until now even for medium-sized problems and limited number of cutting planes. 

\begin{algorithm}[h] 
\footnotesize
	\caption{\textsc{CP--ADMM} for the GPP}\label{Alg:CP_ADMM}
	\SetAlgoLined
	\KwData{{weighted Laplacian matrix $\bar{L}$}, $m_1\geq\ldots \geq m_k$;}
	\KwIn{$UB$, $\epsadmm$, $\epsproj$, $numCuts$, $maxOuterLoops$, $maxTime$;}
	 \KwOut{valid lower bound {$lb(R^p,Z^{p})$} \;}
	Initialization: Set $(R^0,X^0,Z^0)$, $\sigma^0$ and $\gamma$ by using \eqref{ADMM:initialEP} or \eqref{ADMM:initialGBP}. Set $p=0$,  ${\mathcal{T}}=\emptyset$\;
	{Obtain $V$ by applying a QR-decomposition to \eqref{Vp} for the $k$-EP and to \eqref{V2n} for the GBP}\;
	\While{stopping criteria not met}{ 
	   \While{stopping criteria not met}{
	        $R^{p+1} = \mathcal{P}_{\succeq \mathbf{0}} \left ( V^\top \left( {X}^p + \frac{1}{\sigma^p} Z^p \right )V \right )$\;
	        ${X}^{p+1} = \mathcal{P}_{\mathcal{X}_{\mathcal{T}}}\left (VR^{p+1}V^\top- \frac{1}{\sigma^p} \left (\bar{L}+Z^p \right) \right )$ by solving \eqref{ProjectionProblem} using (\ref{AlgCyclicDykstra})\;
	        $ Z^{p+1} = Z^p + \gamma \cdot \sigma^{p} \cdot (X^{p+1}-VR^{p+1}V^\top)$\; 
	        {If adaptive stepsize term is used, update $\sigma^{p+1}$ by using \eqref{updateParam}\;}
	        $p \gets p+1$\;
	   }
   Compute a valid lower bound {$lb(R^p,Z^{p})$}  by using  \eqref{safeLB} \;  
   Identify the violated inequalities and add the $numCuts$ most violated cuts to $\mathcal{T}$\;
     Cluster the cuts in $\mathcal{T}$\;
    }
\end{algorithm}

\subsubsection{Valid lower bounds} \label{sect:lowerbound}

There are existing several ways to obtain valid lower bounds when stopping iterative algorithms earlier, see e.g., \cite{oliveira2018admm,Li2021ASC}. We compute valid lower bounds by exploiting the approach from  \cite{Li2021ASC}.

We use our uniform notation for  the $k$-EP and the GBP
to derive the Lagrangian dual problem for both problems. Let $\mathcal{L}(X,R,Z) := \mathcal{L}_0(X,R,Z)$, see~\eqref{eq:Lagrangianfunction}, denote the Lagrangian function of \eqref{sdp-p2} and \eqref{sdp-bp2} with respect to dualizing $X = VRV^\top$. Then, the Lagrangian dual of \eqref{sdp-p2} and \eqref{sdp-bp2} is
\begin{equation} \label{eq:lagrangiandual2} 
    \begin{aligned}
        \max_{Z\in  {\mathcal S}^{q}}\min_{{X}\in \mathcal{X}_{\mathcal T},R\in \mathcal{R}} \mathcal{L} ({X},R,Z) &
          = \max_{Z\in  {\mathcal S}^q} \left \{ \min_{{X}\in \mathcal{X}_{\mathcal T}}  \langle \bar{L}+Z,{X}\rangle -
      \trace(R)\lambda_{\max} \left (V^\top Z V \right) \right \},\\
    \end{aligned}
\end{equation}
where $\lambda_{\max}(V^\top Z V)$ is the largest eigenvalue of $V^\top Z V$, and $q$ is the appropriate order of the positive semidefinite cone. 
In~\eqref{eq:lagrangiandual2} we exploit the well-known Rayleigh principle. 
It follows from~\eqref{eq:lagrangiandual2}  that for any $Z\in  {\mathcal S}^q$ one can obtain a valid lower bound by computing:
\begin{align} \label{safeLB}
lb(R,Z)=\min_{{X}\in \mathcal{X}_{\mathcal T}}  \langle \bar{L}+Z,{X}\rangle - \trace(R)\lambda_{\max}(V^\top Z V).
\end{align}
Since the minimization problem in \eqref{safeLB} is a linear programming problem, the computation of valid lower bounds is efficient.

\subsubsection{Stopping criteria for the CP--ADMM algorithm} \label{sec:stopp}

We use different stopping criteria for the inner and outer while-loops in Algorithm \ref{Alg:CP_ADMM}.
The following measure is used as one of the stopping criteria for the inner while-loop:
\begin{align*}
 \max \left \{ \frac{\|{X}^p - VR^p V ^\top \|_F}{1 + \|{X}^p\|_F},  \sigma \frac{\|{X}^{p+1}-{X}^p\|_F}{1 + \|{Z}^p\|_F} \right \} < \epsadmm,   
\end{align*}
where ${\epsadmm}$ is the prescribed tolerance precision. We also stop the inner while-loop when $maxTime$ is reached.

The Dykstra's projection algorithm (\ref{AlgCyclicDykstra}) stops when $\|X^{q+1}-X^q \|_F< \epsproj $ for a given input parameter $\epsproj$.

We consider the following types of stopping criteria for the 
outer while-loop:
\begin{itemize}
\item The algorithm stops if the gap between a valid lower bound, that is rounded up to the closest integer, and a given upper bound $UB$ is closed.

\item The algorithm stops  if an improvement in lower bounds between two consecutive outer loops  is less than the prescribed threshold, i.e., $0.001$.
\item The algorithm stops if the number of new cuts to be added in the next outer loop is small, i.e., $<0.25n$.
    \item The algorithm stops if the maximum number of outer loops $maxOuterLoops$ is reached.
    \item The algorithms stops immediately if the maximum computation time $maxTime$ is reached. 
\end{itemize} 
We specify the values of the input parameters in  Section~\ref{sect:numerics}.

\subsubsection{Efficient ingredients of the CP--ADMM algorithm}

Algorithm~\ref{Alg:CP_ADMM} is efficient due to the following ingredients:
\begin{enumerate}
    \item Warm starts. After identifying new cuts we start the new ADMM iterate from the last obtained  triple $(R^p,X^p,Z^p)$. Observe that there is no warm start strategy for an interior point method.
    \item Scaling of data.
It is known that the performance of  a first order method can be improved by appropriate scaling of data. Therefore, 
we scale the objective by a scalar $S\in \R$ that depends on the problem and its size. Namely, for the $k$-EP  we set
$S=1/{\| L\|_F}$ for $n \leq 400$, 
$S={k}/{\| L\|_F}$ for $ 400< n \leq 800$, and
$S={n}/{(k \| L\|_F )}$ otherwise.
For the GBP we use $S=1$.
The values for $S$ are obtained by extensive  numerical tests.
\item Clustering. A crucial ingredient for improving the performance of Dykstra's projection algorithm is clustering cuts, see Section~\ref{sect:Dykstra}.

\item Separation. We  introduce a probabilistic independent set separation method to separate independent set inequalities, see Algorithm~\ref{alg:separate_clique_prob} in Section~\ref{sect:cuts}.
\end{enumerate}

\textcolor{black}{When we run the CP--ADMM without any cutting planes (i.e., $\mathcal{T} = \emptyset$), the bottleneck of the code is the projection onto the positive semidefinite  cone. When we start adding cuts, Dykstra's algorithm starts taking over the major part of the computation time. }

\section{Valid cutting planes, their projectors and separators}\label{sect:cuts}

In this section we consider various families of cutting planes that strengthen the DNN relaxations for the $k$-EP and GBP. In the light of adding them in the cutting-plane augmented Lagrangian algorithm of Section~\ref{sect:admm}, we present for each cut type a polyhedral set $\mathcal{H}_t$ induced by the cut (and, possibly, a subset of the constraints from the corresponding DNN relaxation). We show how to explicitly project a matrix onto these polyhedral sets. The efficient separation of these cut types is also considered.

In total we consider three types of cutting planes: two for the $k$-EP and one for the GBP. 

\subsection{Triangle inequalities for the  $k$-EP} \label{subsect:triangle}

Let us consider the relaxation $(DNN_{EP})$, see \eqref{eq:relax_dnn} for the $k$-equipartition problem. 
\textcolor{black}{Marcotorchino~\cite{Marcotorchino} as well as Gr\"otschel and Wakabayashi~\cite{groetschel1989cuttingplane}} 
observe  that \textcolor{black}{the linear relaxation of the $k$-equipartition problem} can be strengthened by adding the triangle  inequalities:
\begin{align} \label{ineq:triangle}
Y_{ij} + Y_{il} \leq 1 + Y_{jl} \quad \text{for all triples } (i,j,l), i \neq j, j \neq l, i \neq l.
\end{align}
 For a given triple $(i,j,l)$ of distinct vertices, the triangle constraint \eqref{ineq:triangle} ensures that if 
$i$ and $j$ are in the same set of the partition and so are $i$ and $l$, then also $j$ and $l$ have to belong to the same  set of the partition.
\textcolor{black}{Karisch and Rendl~\cite{Karisch98semidefinite} use these inequalities to strengthen $(DNN_{EP})$.}

To obtain the equivalent facially reduced relaxation \eqref{sdp-p}, we apply the linear transformation $X = Y - \frac{1}{k}{\mathbf J}_n$, see Section~\ref{sect:equipartition}. As we apply our cutting-plane algorithm on this latter relaxation, we also perform this transformation on the triangle inequalities. The transformed cuts are as follows: 
\begin{align}\label{ineq:trianglehat}
X_{ij} + X_{il} \leq \frac{k-1}{k} + X_{jl} \quad \text{for all triples } (i,j,l), i\neq j, j \neq l, i \neq l.
\end{align}
Observe that there exist $3\binom{n}{3}$ triangle inequalities. 
\medskip \\
To incorporate the cutting planes \eqref{ineq:trianglehat} into our cutting-plane augmented Lagrangian algorithm, we define for each cut a polyhedral set that is induced by the cut. For each triple $(i,j,l)$ we define the polyhedron $\mathcal{H}_{ijl}^\Delta \subseteq \mathcal{S}^n$ as follows:  
\begin{align} \label{polyhedra:triangle}
\mathcal{H}_{ijl}^\Delta := \left\{ X \in \mathcal{S}^n \, : \, \,X_{ij} + X_{il} \leq \frac{k-1}{k} + X_{jl} \right\}. 
\end{align}
Let $\mathcal{P}_{\mathcal{H}^\Delta_{ijl}}\colon \mathcal{S}^n \to \mathcal{S}^n$ denote the operator that projects a matrix in $\mathcal{S}^n$ onto $\mathcal{H}^\Delta_{ijl}$. As the idea behind Dykstra's cyclic projection algorithm suggests, this projector can be characterized by a closed form expression. 

\begin{lemma} \label{Lem:triangle}
Let $M \in \mathcal{S}^n$ and let $\hat{M} := \mathcal{P}_{\mathcal{H}^\Delta_{ijl}}(M)$. If $M \in \mathcal{H}^\Delta_{ijl}$, then $\hat{M} = M$. If $M \notin \mathcal{H}^\Delta_{ijl}$, then $\hat{M}$ is such that
\begin{align*}
\hat{M}_{pq} = \begin{cases} \frac{2}{3} M_{ij} - \frac{1}{3} M_{il} + \frac{1}{3} M_{jl} + \frac{1}{3} - \frac{1}{3k} & \text{if $(p,q) \in \{(i,j), (j,i)\}$,} \\
-\frac{1}{3}M_{ij} + \frac{2}{3} M_{il} + \frac{1}{3} M_{jl} + \frac{1}{3} - \frac{1}{3k} & \text{if $(p,q) \in \{(i,l), (l,i)\}$,} \\
\frac{1}{3}M_{ij} + \frac{1}{3}M_{il} + \frac{2}{3} M_{jl} - \frac{1}{3} + \frac{1}{3k} & \text{if $(p,q) \in \{(j,l), (l,j)\}$,} \\
M_{pq} & \text{otherwise.}
\end{cases}
\end{align*}
\end{lemma}
\begin{proof} See Appendix~\ref{App:prooftriangle}. 
\end{proof}
Identifying the most violated inequalities of the form \eqref{ineq:trianglehat} can be done by a complete enumeration. This separation can be done in $O(n^3)$. 

\subsection{Independent set inequalities for the $k$-EP} \label{subsect:clique}
\textcolor{black}{Chopra and Rao~\cite{chopra1993partition} introduced a further type of inequalities} that are valid for \textcolor{black}{the linear relaxation of the $k$-equipartition problem}, namely
\begin{align} \label{ineq:clique}
\sum_{i, j \in I, i < j} Y_{ij} \geq 1 \quad \text{for all $I \subseteq V$ with $|I| = k+1$},
\end{align}
which are known as the independent set inequalities. 
\textcolor{black}{These inequalities have also been used for the SDP relaxation in the work of Karisch and Rendl~\cite{Karisch98semidefinite}.}
The intuition behind these constraints is that for all subsets of $k + 1$ nodes, there must always be two nodes that are in the same set of the partition. Thus, the graph with adjacency matrix $Y$ has no independent set of size $k+1$. 

Using the linear transformation $X =  Y - \frac{1}{k}{\mathbf J}_n$, we obtain the following equivalent inequalities that are valid for the facially reduced relaxation \eqref{sdp-p}: 
\begin{align} \label{ineq:cliquehat}
\sum_{i, j \in I, i < j} X_{ij} \geq \frac{1 - k}{2} \quad \text{for all $I$ with $|I| = k+1$}. 
\end{align}
Observe that there are $\binom{n}{k + 1}$ independent set inequalities. 

 Let us define for each set $I \subseteq V$ with $|I| = k+1$ a polyhedral set $\mathcal{H}_{I}^{IS} \subseteq \mathcal{S}^n$ that is induced by the cut, i.e.,
\begin{align}\label{polyhedra:cut}
\mathcal{H}_I^{IS} := \left\{ X \in \mathcal{S}^n \, : \, \,  \sum_{i, j \in I, i < j} X_{ij} \geq \frac{1 - k}{2} \right\}. 
\end{align}
We let $\mathcal{P}_{\mathcal{H}^{IS}_I}\colon \mathcal{S}^n \to \mathcal{S}^n$ denote the projector onto the polyhedron $\mathcal{H}^{IS}_I$. This projection can be performed explicitly, as shown by the following result. 
\begin{lemma} \label{Lem:clique}
Let $M \in \mathcal{S}^n$ and let $\hat{M} := \mathcal{P}_{\mathcal{H}_I^{IS}}(M)$. If $M \in \mathcal{H}^{IS}_I$, then $\hat{M} = M$. If $M \notin \mathcal{H}^{IS}_I$, then $\hat{M}$ is such that
\begin{align*}
\hat{M}_{pq} = \begin{cases} M_{pq} - \frac{k-1}{k(k+1)} - \frac{2}{k(k+1)} \sum_{i< j, i,j \in I}M_{ij}, & \text{if $p, q \in I$, $p \neq q$,} \\
M_{pq}, & \text{otherwise.}
\end{cases}
\end{align*}
\end{lemma}
\begin{proof}
See Appendix~\ref{App:proofclique}. 
\end{proof}
In order to find the most violated inequalities of type~\eqref{ineq:cliquehat}, we need a separator for independent set inequalities. It is known that exact separation of these inequalities is NP-hard~\cite{eisenblaetter}. Complete enumeration leads to a running time of $O(n^{k+1})$, which is computationally tractable only for $k = 2$ and $k = 3$. For larger $k$, we apply a combination of two separation heuristics to identify violated inequalities. First, we apply the deterministic separation heuristic from  \citet{anjos2013solving}. This method efficiently generates at most $n$ inequalities, which turn out to be effective as numerical experiments in~\cite{anjos2013solving} suggest.

On top of the heuristic from~\cite{anjos2013solving}, we also introduce a probabilistic independent set inequality separation heuristic. Although this algorithm relies on the same idea as the deterministic heuristic from~\cite{anjos2013solving}, the greedy selection of a new vertex to add in the set $C$ is randomized with probabilities inversely proportional to their values in the current solution matrix $X$. A pseudo-code of this heuristic is given in Algorithm~\ref{alg:separate_clique_prob}. The parameter $N_R$ corresponds to the number of repetitions, while $\varepsilon > 0$ is a sensitivity parameter. Low values of $\varepsilon$ lead to very sensitive behaviour with respect to differences in the current solution $X$, while the selection eventually resembles a uniform distribution when $\varepsilon$ is increased. The advantage of this randomization is that the combination of both heuristics can yield more than $n$ violated independent set inequalities.

\begin{algorithm}[H] 
\footnotesize
	\caption{Probabilistic  separation method for independent set inequalities}\label{alg:separate_clique_prob}
	\SetAlgoLined
	\KwData{the number of partitions $k$, the size of graph $n$;}
	\KwIn{ output matrix $X$ from the ADMM, number of repetitions $N_R$, sensitivity parameter $\varepsilon > 0$;}
	\KwOut{a collection of violated distinct independent set inequalities  $\mathcal{C}$, its violation vector $v$\;}
	Initialization: $Y = X + \frac{1}{k}\mathbf{J}_n$, $\mathcal{C} = \emptyset$\;
	\For{  $r \in [N_R] $}
	{   Choose vertex $v$ uniformly at random from $[n]$\;
	   $C \leftarrow \{ v \} $\;
	   $S \leftarrow [n] \setminus \{v \}$\;
	   \For{$l \in [k]$}{
	        Define $p_i := \sum_{j \in C}{y_{ij}} + \varepsilon$ for all $i \in S$\;
	        Define $q_i := \frac{(1/p_i)}{\sum_{i \in S}(1/p_i)}$ for all $i \in S$\;
	        Randomly select vertex $i \in S$ according to probability mass function $\{q_i\}_{i \in S}$\; 
	        $C \leftarrow C \cup \{ i\}, \,\, S \leftarrow S \setminus \{i\}$\;
	   }
     \If{$C \notin \mathcal{C}$}{ $\mathcal{C}\leftarrow \mathcal{C} \cup \{ C \}$  \;}
     }
\end{algorithm}

\subsection{BQP inqualities for the GBP} \label{subsect:bqp}

We now consider the relaxation $(DNN_{BP})$, see \eqref{DNNBP}.
The relaxation $(DNN_{BP})$ can be further strengthened by adding the following inequalities: 
\begin{eqnarray}
0 \leq X_{ij} \leq X_{ii} \label{BQP1} \\[1ex]
X_{ii} + X_{jj} \leq 1 + X_{ij} \label{BQP2} \\[1ex]
X_{il} + X_{jl} \leq X_{ll} + X_{ij} \label{BQP3} \\[1ex]
X_{ii} + X_{jj} + X_{ll} \leq X_{ij} + X_{il} + X_{jl} +1, \label{BQP4} 
\end{eqnarray}
where $X=(X_{ij})\in {\mathcal S}^{2n+1}$ and  $ 1\leq  i,j,l \leq 2n$, $i\neq j$, $i\neq l$, $j\neq l$.
The  inequalities  \eqref{BQP1}--\eqref{BQP4}  are facet defining
inequalities of the boolean quadric polytope~\cite{Padberg}. Wolkowicz and Zhao~\cite{wokowicz-zhao1999} prove that the inequalities \eqref{BQP1} and \eqref{BQP2} are already implied by the constraints in \eqref{eq:ZW}.  Moreover, preliminary  numerical results show that the inequalities~\eqref{BQP3} make larger improvements in the bounds when added to the DNN relaxation than the inequalities~\eqref{BQP4}. 
Therefore, we consider only the constraints~\eqref{BQP3} within our algorithm.

Different from the SDP relaxation of the $k$-EP, the polyhedral set $\mathcal{X}_{BP}$ is a subset of the lifted space $\mathcal{S}^{2n+1}$. As a result, the polyhedral part induced by a BQP cut of the form \eqref{BQP3} is also a subset of $\mathcal{S}^{2n+1}$. For each triple $(i,j,l)$ with $ 2\leq  i,j,l \leq 2n +1$, $i\neq j$, $i\neq l$, $j\neq l$, we define the following polyhedron: 
\begin{align}\label{polyhedra:bqp}
\mathcal{H}^{BQP}_{ijl} := \left\{  \begin{pmatrix}
1 & (x^1)^\top & (x^2)^\top \\
x^1 & X^{11} & X^{12} \\
x^2 & X^{21} & X^{22}
\end{pmatrix} \in \mathcal{S}^{2n+1}  :   \begin{aligned} X = \begin{pmatrix}
X^{11} & X^{12} \\ X^{21} & X^{22}
\end{pmatrix}, \, X_{il} + X_{jl} \leq X_{ll} + X_{ij} \\
\diag(X^{11}) = x^1, \diag(X^{22}) = x^2,  x^1 + x^2 = {\mathbf 1}_n \end{aligned} \right\}.
\end{align}
The polyhedron $\mathcal{H}_{ijl}^{BQP}$ is not only induced by the BQP cut, it also contains a subset of the constraints of the relaxation~\eqref{DNNBP}. This idea is inspired by the approach in~\cite{deMeijerSotirov}, where the inclusion of additional constraints in each polyhedron in Dykstra's algorithm speeds up the convergence. Since the structure of $\mathcal{H}_{ijl}^{BQP}$ must remain simple enough to project onto it via a closed form expression, it is impractical to add all constraints from~\eqref{DNNBP}. The set $\mathcal{H}_{ijl}^{BQP}$ is chosen such that we are still able to project onto it explicitly. 

Let $\mathcal{P}_{\mathcal{H}^{BQP}_{ijl}}: \mathcal{S}^{2n+1} \rightarrow \mathcal{S}^{2n+1}$ denote the projector onto $\mathcal{H}^{BQP}_{ijl}$. Given that the matrix that is projected already satisfies the constraints $\diag(X^{11}) = x^1, \diag(X^{22}) = x^2$ and $x^1 + x^2 = \mathbf{1}_n$, which is always the case in our implementation, this projector is specified by the result below.

\begin{lemma} \label{Lem:bqp}
Let $M = \begin{pmatrix}
1  & \diag(M^{11})^\top & \diag (M^{22})^\top \\
\diag(M^{11}) & M^{11} & M^{12} \\
\diag (M^{22}) & M^{21} & M^{22}
\end{pmatrix} \in \mathcal{S}^{2n+1}$
be such that $\diag(M^{11})+\diag (M^{22})=\mathbf{1}_n$
and let
$\hat{M} := \mathcal{P}_{\mathcal{H}^{BQP}_{ijl}}(M)$. If $M_{il} + M_{jl} \leq M_{ij} + \frac{1}{6}M_{ll} + \frac{1}{3}M_{1l} - \frac{1}{6}M_{l^*l^*} - \frac{1}{3}M_{1l^*} + \frac{1}{2}$, then
\begin{align*}
    \hat{M}_{pq} = \begin{cases}
    \frac{1}{6}M_{ll} + \frac{1}{3}M_{1l} - \frac{1}{6}M_{l^*l^*} - \frac{1}{3}M_{1l^*} + \frac{1}{2} & \text{if $(p,q) \in \{(l,l), (1,l), (l,1)\},$} \\
   -\frac{1}{6}M_{ll} - \frac{1}{3}M_{1l} + \frac{1}{6}M_{l^*l^*} + \frac{1}{3}M_{1l^*} + \frac{1}{2}& \text{if $(p,q) \in \{(l^*,l^*), (1,l^*),(l^*,1)\},$} \\
    M_{pq} & \text{otherwise.}
    \end{cases}
\end{align*}
Otherwise, $\hat{M}$ is such that
\begin{align*}
\footnotesize
\hat{M}_{pq} = \begin{cases} \frac{7}{10} M_{il} - \frac{3}{10} M_{jl} + \frac{3}{10} M_{ij} + \frac{1}{20}M_{ll} + \frac{1}{10}M_{1l} - \frac{1}{20}M_{l^*l^*} - \frac{1}{10}M_{1l^*}  + \frac{3}{20} & \text{if $(p,q) \in \{(i,l), (l,i)\}$,} \\
 -\frac{3}{10} M_{il} + \frac{7}{10} M_{jl} + \frac{3}{10} M_{ij} + \frac{1}{20}M_{ll} + \frac{1}{10}M_{1l} - \frac{1}{20}M_{l^*l^*} - \frac{1}{10}M_{1l^*}  + \frac{3}{20} & \text{if $(p,q) \in \{(j,l), (l,j)\}$,} \\
  \frac{3}{10} M_{il} + \frac{3}{10} M_{jl} + \frac{7}{10} M_{ij} - \frac{1}{20}M_{ll} - \frac{1}{10}M_{1l} + \frac{1}{20}M_{l^*l^*} + \frac{1}{10}M_{1l^*}  - \frac{3}{20} & \text{if $(p,q) \in \{(i,j), (j,i)\}$,} \\
  \frac{1}{10} M_{il} + \frac{1}{10} M_{jl} - \frac{1}{10} M_{ij} + \frac{3}{20}M_{ll} + \frac{1}{10}M_{1l} - \frac{3}{20}M_{l^*l^*} - \frac{3}{10}M_{1l^*}  + \frac{9}{20} & \text{if $(p,q) \in  \begin{aligned}[t] \{ &(l,l), 
  (1,l), \\&(l,1)\}\end{aligned} $} \\
  -\frac{1}{10} M_{il} - \frac{1}{10} M_{jl} + \frac{1}{10} M_{ij} - \frac{3}{20}M_{ll} - \frac{3}{10}M_{1l} + \frac{3}{20}M_{l^*l^*} + \frac{3}{10}M_{1l^*} + \frac{11}{20} & \text{if $(p,q) \in \begin{aligned}[t] \{&(l^*,l^*), (1,l^*),\\ &(l^*,1)\} \end{aligned}$} \\
M_{pq} & \text{otherwise.}
\end{cases}
\end{align*}
where $l^*$ is obtained from $l$ by $l^* := 2 + (l+ n - 2) \bmod 2n$.
\end{lemma}
\begin{proof} See Appendix~\ref{App:proofbqp}. 
\end{proof}
Separating the BQP inequalities \eqref{BQP3} can be done in $O(n^3)$ by complete enumeration.

\section{Numerical results}\label{sect:numerics}

\textcolor{black}{We implemented our algorithm CP--ADMM in Matlab. For efficiency some separation routines have been coded in~C.} 
In order to evaluate the quality of the bounds and the run times to compute these bounds, we test our algorithms on various instances from the literature. All experiments were run on 
 an  Intel Xeon, E5-1620, 3.70~GHz with 32~GB memory.
To compute valid lower bounds after each run of the inner while-loops we use Mosek~\cite{Mosek}. Note that the computation of a valid bound after the inner while-loops is necessary for checking the stopping criteria.

We now describe the data sets used in our evaluation. Most of these instances were also considered in~\cite{armbruster2007} and~\cite{HelmbergEtAl}.
\begin{itemize}

    \item  $G_{|V|,|V|p}$ and $U_{|V|,|V|{\pi d^2}}$: randomly generated graphs by Johnson et al.~\cite{johnson1989optimization}. 
    \begin{itemize}
        \item $G_{|V|,|V|p}$: graphs $G = (V,E)$, with $|V| \in \{124, 250, 500, 1000\}$ and four individual edge probabilities $p$. 
        These probabilities were chosen depending on $|V|$, so that the average expected degree of each node was approximately $|V| p \in \{2.5, 5, 10, 20\}$.
        
	    \item $U_{|V|,|V|{\pi d^2}}$: graphs $G=(V,E)$, with $|V| \in \{500,1000\}$ with distance value $d$ such that $|V|{\pi d^2} \in \{ 5, 10, 20, 40\}$.
	    To form such a graph $G=(V,E)$, one chooses $2|V|$ independent numbers uniformly from the interval $(0,1)$ and views them as coordinates of $|V|$ nodes on the unit square. An edge is inserted between two vertices if and only if their Euclidean distance is at most $d$.

    \end{itemize}
    
    \item Mesh graphs from~\cite{de1993graph,ferreira1998}: Instances from finite element meshes; all nonzero edge weights are equal to one.  Graph names begin with `mesh', followed by the number of vertices and the number of edges.

    \item KKT graphs: These instances originate from nested bisection approaches for solving sparse symmetric linear systems. Each instance consists of a graph that represents the support structure of a sparse symmetric linear system, for details see~\cite{helmberg2004cutting}.
 
     \item Toroidal 2D- and 3D-grid graphs arise in physics when computing ground states for Ising spinglasses, see e.g.,~\cite{helmberg2004cutting}. They are generated using the \texttt{rudy} graph generator~\cite{rudygenerator}: 
    \begin{itemize}
        \item spinglass2pm\_$n_r$: A toroidal 2D-grid for a spinglass model with weights $\{+1,-1\}$.
         The grid has size $n_r \times n_r $, i.e., $|V| = n_r^2$.
         The percentage of edges with negative weights is $50~\%$. 
        \item spinglass3pm\_$n_r$: A toroidal 3D-grid for a spinglass model with weights $\{+1,-1\}$.
         The grid has size $n_r \times n_r  \times n_r$, i.e., $|V| =n_r^3$.
         The percentage of edges with negative weights is $50~\%$.
    \end{itemize}
\end{itemize}

\subsection{Numerical results for the $k$-EP}

In the CP--ADMM, see Algorithm~\ref{Alg:CP_ADMM}, we input an upper bound $UB$ and the parameters $maxTime$, $numCuts$, $maxOuterloops$, $\epsproj$, and $\epsadmm$.  We also require the bounds $\sigma_{\min}$, and $\sigma_{\max}$ for the adaptive stepsize term. The setting of these parameters is as follows:
\begin{itemize}
    \item As an upper bound we input the values we obtained by heuristics or the optimal solution given in the literature.
    \item The maximal number of cuts added in each outer while-loop, $numCuts$, is $3n$ for graphs with $n\leq 300$ and $5n$ when $n>300$. These values are determined by preliminary tests, see also Appendix~\ref{App:numericalresults}. 
    
    \item  The maximal number of outer while-loops is~30 for instances with $n\leq 300$,  and~10 when $n>300$. 
    \item The precision for Dykstra's projection algorithm $\epsproj$ is set to $10^{-4}$.
    \item The inner precision $\epsadmm$ is $10^{-4}$ in the last iteration and $10^{-3}$ in all previous loops.
    \item The maximum computation time $maxTime$ is set to 2~hours.
    \item The bounds $\sigma_{\min}$ and $ \sigma_{\max}$ for  $\sigma^p$ are $10^{-5}$ and $  10^3$, respectively.
\end{itemize}
In each outer while-loop we separate $MaxIneq$ violated inequalities. We experimented with two strategies:
One strategy is to first add violated triangle inequalities and, in case less than $MaxIneq$ violated triangle inequalities are found, we add violated independent set inequalities.
The other strategy is to mix the two kinds of cuts and search for violated triangle and independent set inequalities together.
The experiments showed that the latter strategy obtains better results, i.e., better bounds within the same time. Therefore, in the final setting we search for the most violated inequalities from both, the triangle and independent set inequalities. 

The separation of triangle inequalities is done by complete enumeration.
Searching for independent set inequalities is also done by complete enumeration if $k\in \{2,3\}$. For $k \in \{4,5\}$ we apply the heuristic from~\cite{anjos2013solving} and Algorithm~\ref{alg:separate_clique_prob}, as explained in Section~\ref{subsect:clique}. 

\bigskip
In Table~\ref{table:compareAll} we compare the eigenvalue lower bound by Donath and Hoffman~\cite{Donath1973LowerBF} (denoted by DH) with the lower bound when computing the DNN relaxation as well as the DNN relaxation with additional cuts on selected graphs from our testbed.
We do not include other bounds from the literature, such as  bounds from \cite{Hager2013AnEA}, as they are weaker than the DH bound.
The DH bound is {\color{black}  obtained by solving the corresponding SDP relaxation using Mosek~\cite{Mosek}.} 
The numbers in the table show that the DNN bound significantly improves over the DH bound. {\color{black}Moreover, our ADMM algorithm requires less time and memory to compute  the DNN bound than Mosek to compute the DH bound.} Adding cuts to the DNN gives an even more substantial improvement. Hence, including triangle and independent set inequalities is much stronger than including non-negativity constraints only. As the DH bound is not competitive with the DNN bound and far from the DNN+cuts bound, we will not include it in the subsequent presentation of the numerical results.

Computing an equipartition using commercial solvers is out of reach unless the graphs are extremely sparse. For instances in Table~\ref{table:compareAll}, Gurobi {\color{black} (with the default settings)} solves the small and very sparse instances within a few seconds, but obtains a gap of more than 40 \% after 2 hours for larger, but still reasonably sparse graphs. E.g., for $G_{250,5}$, a graph on  250 vertices with density 2 \%, the gap after two hours is more than 40 \%. For $G_{250,20}$, a graph on 250~vertices and density 8 \%, the gap is even more than 80 \% after two hours.
This limited power of commercial solvers to tackle the $k$-EP is also observed in \cite{Beluc}. Hence, we omit the comparison to Gurobi or other LP-based solvers in the tables.

 \begin{table}[htp!]
    \centering
    \footnotesize
    \begin{tabular}{rrrr|rrr}
    \hline 
    Graph & $n$ & $k$ & ub & {DH} & {DNN} &{DNN + Cuts}\\
    \hline 
    mesh.70.120 & 70 &2	&7 & 1.93 & 2.91 &6.02 \\
    KKT.lowt01 & 82 &2 &13 &2.47 & 4.88 & 12.43 \\
      mesh.148.265 & 148 & 4 &22 & 5.46& 8.13&21.23 \\
    $G_{124,2.5}$ &124 &2& 13&  4.59 & 7.29 & 12.01\\ 
    $G_{124,10}$  &124 & 2 & 178 & 138.24   &152.86 & 170.88 \\
    $G_{124,20}$  &124 & 2 & 449 & 403.08  &418.67 & 439.96\\
    $G_{250,2.5}$ & 250 & 2 & 29 & 10.99  &15.16 & 28.30 \\
    $G_{250,5}$   &250 & 2 & 114 & 70.21 & 81.52 &105.00\\  
    $G_{250,10}$  &250 & 2 & 357 & 280.25  &303.02 & 330.40\\ 
    \hline
    \end{tabular}
    \caption{Comparison of different lower bounds}
    \label{table:compareAll}
\end{table}

In Tables~\ref{tab:hager_table9_tol=10-3_mixcuts=1} to~\ref{tab:LargeGraph_GU_5n_k=5_mixcuts=1_tol=10-3} we give the details of our numerical results. In the first 4 columns we list the name of the instance, the number of vertices $n$, the partition number $k$ and an upper bound. 
The upper bounds are obtained by heuristics or the optimal solution from the literature. 

In columns~5 and~6 the lower bound (rounded up) and the time when solving the DNN relaxation are given. Finally, in the remaining columns we display the results when adding cuts to the DNN relaxation: we report the rounded up lower bounds, computation time and the improvement (in \%) on the rounded up lower bounds with respect to the DNN relaxation without cuts.
We decided to report the improvement with respect to the rounded values, since otherwise for small numbers the percentages are incredibly huge. E.g., for instance $U_{1000,5}$ and $k=4$, the value of the DNN bound is 0.17, the DNN+cuts bound is 2.45, giving an improvement of 1,341.2~\%.
The rounded up values are 1 and 3, respectively, giving a 200~\% improvement, which reflects the situation much better.

In columns~10 and~11 we list the number of triangle cuts and independent set cuts present when stopping the algorithm. In the last two columns, the number of iterations of the ADMM and the number of outer while-loops is reported.

As can be observed in all tables, the bounds improve drastically when adding triangle and independent set inequalities to the DNN relaxation while the time for computing these bounds is still reasonable. {\color{black}We remark here that we can stop the CP--ADMM algorithm at any time and provide a valid lower bound.}

The results show that the CP--ADMM  takes much more of the triangle  inequalities when computing bounds for the $k$-equipartition problem. Remember that we search for triangle and independent set inequalities together. Hence,
there are more triangle inequalities with a large violation which means that the triangle inequalities contribute more to the strength of the bound than the independent set inequalities.

We discuss the results in more detail in the subsequent sections. 

\subsubsection{Detailed results for $k=2$}
In Tables~\ref{tab:hager_table9_tol=10-3_mixcuts=1} and~\ref{tab:LargeGraph_GU_5n_mixcuts=1_tol10-3} we report results for $k=2$. Table~\ref{tab:hager_table9_tol=10-3_mixcuts=1} includes graphs with up to 274~vertices, the DNN bound for these graphs can be computed within a few seconds. After adding in total some 500 up to 16,000 triangle inequalities and additional 300 up to roughly 7,000 independent set inequalities, the bound improves between 4.55~\% and 300~\%. In several cases, the bound closes the gap to the best known upper bound. Otherwise, the algorithm stops due to the little improvement of the bounds in consecutive outer while-loops. 
The time for computing these bounds ranges from a few seconds to 17~minutes.

In Table~\ref{tab:LargeGraph_GU_5n_mixcuts=1_tol10-3} we consider larger graphs with 500 and 1,000 vertices. The DNN bound can be computed for these graphs within 12~minutes. After adding triangle and independent set inequalities, the bound improves up to 200~\%  in a running of 40~seconds up to 2~hours. On the $G$ graphs one observes that the improvement of the bound when adding cuts gets more significant as the graphs get sparser. Note that for all these instances we stop because the number of outer while-loops is reached or the improvement of the bounds in consecutive outer while-loops is too small.

We give further results for graphs from the literature in Table~\ref{tab:PlanarGraph_3n_mixcuts=1_tol=10-3} in Appendix~\ref{App:numericalresults}. For most of these we prove optimality of the best found $k$-equipartition, confirming the high quality of our lower bounds.

\subsubsection{Detailed results for $k>2$}
{\color{black}In Tables~\ref{tab:hager_table9DiffK_tol=10-3_mixcuts=1} to~\ref{tab:LargeGraph_GU_5n_k=5_mixcuts=1_tol=10-3} we report results for $k>2$.}
As in the case for $k=2$, the bounds can be significantly improved while the time for obtaining the bounds is still reasonable. 
However, we close the gap for fewer instances as for $k=2$.
The improvement for the larger graphs after adding cuts to the DNN relaxation is up to 200~\% for $k=4$ and up to 300~\% for $k=5$. 

For smaller graphs the CP--ADMM stops because the improvement of the lower bound compared to the previous iteration is below the threshold~0.001, see Section~\ref{sec:stopp}.The largest improvement in the DNN+cuts bound w.r.t.~the DNN bound  is 500~\%.
Again, we observe that for graphs with more than 250 vertices, the algorithm typically stops because the maximum number of outer loops is reached.

\subsection{Numerical results for bisection}

In Algorithm~\ref{Alg:CP_ADMM}, we input an upper bound $UB$ and the parameters  $numCuts$,  $maxOuterloops$,    $\epsproj$, $\epsadmm$, and $maxTime$. The setting of these parameters is as follows.
\begin{itemize}
    \item As an upper bound we input the numbers we obtained by heuristics.
    \item The maximal number of cuts added in each outer while-loop, $numCuts$, is $3n$ for graphs with $n\leq 300$ and $5n$ when $n>300$.
    \item  The maximal number of outer while-loops is~30 for instances with $n\leq 300$,  and~10 when $n>300$. 
    \item The precision for Dykstra's projection algorithm $\epsproj$ is set to $10^{-6}$.
    \item The inner precision $\epsadmm$ is $10^{-5}$ in the first and last inner while-loop, and $10^{-4}$ in all other  inner while-loops. 
    \item The maximum computation time $maxTime$ is set to 2~hours.
\end{itemize}
The parameters differ from the experiments for the $k$-equipartition problem since for the bisection problem the size of the matrix in the SDP is of order $2n+1$, i.e., more than double the size than for the $k$-equipartition problem.
Violated BQP inequalities~\eqref{BQP3}  are found by an enumeration search. 

We take $m_1 = \lceil np \rceil$ where $n$ is the number of vertices in the graph and for $p$ we choose a number out of $\{ 0.6, 0.65, 0.7 \}$.

The results are given in
Table~\ref{tab:TableHager_finaltol=0.001_bisection} for smaller graphs and in Table~\ref{tab:Table_largeG_Spinglass2pm_bisection} for larger graphs. 
Similar as for the $k$-equipartition problem, we observe a significant improvement of the bound after adding inequalities. 
For graphs with vertices up to 274, see Table~\ref{tab:TableHager_finaltol=0.001_bisection}, the improvement of the DNN+cuts bound over the DNN bound ranges between 2.23~\% and 200~\% after adding   up to 16,000  BQP inequalities. For five of the graphs in Table~\ref{tab:TableHager_finaltol=0.001_bisection} the algorithm stopped because the gap was closed, for four graphs the improvement of the lower bound was only minor and four times the number of violated cuts found was too small. Only for one graph the algorithm stopped due to the time limit.
For the larger graphs we typically stop because of the time limit of 2~hours, see Table~\ref{tab:Table_largeG_Spinglass2pm_bisection}. For those graphs the DNN+cuts bound improves over the DNN bound between 0.44~\% and 17.32~\% after adding between 8,000 and 24,200 BQP inequalities.

One can observe that the results are somewhat weaker than for the $k$-equipartition problem. 
Note that the nature of the cuts added to the DNN relaxations for the $k$-EP and GBP differs, and the size of  matrix variables in the GBP relaxations is more than twice larger than in the case of the $k$-EP for the same graph. 
\textcolor{black}{For large graphs, we are able to add roughly up to 50,000 cuts for the $k$-EP and 24,000 cuts for the GBP within a time span of 2 hours.}

\section{Conclusions}
This study aims to investigate and expand the boundary of obtaining strong DNN bounds with additional cutting planes for large graph partition problems. Due to memory requirements, state-of-the-art interior point methods are only capable of solving medium-size SDPs and are not suitable for handling lots of polyhedral cuts. We overcome both difficulties by utilizing a first order method in a cutting-plane framework. Our approach focuses on two variations of the graph partition problem: the $k$-equipartition problem and the graph bisection problem.

We first derive DNN relaxations for both problems, see~\eqref{eq:relax_dnn} and \eqref{DNNBP}, then we apply facial reduction to obtain strictly feasible equivalent relaxations, see~\eqref{sdp-p} and \eqref{ZWslater}, respectively. To prove the minimality of the face of the DNN cone containing the feasible set of \eqref{DNNBP}, we exploit the dimension of the bisection polytope, see Theorem~\ref{Thm:DimConvF}. After facial reduction, both relaxations enclose a natural splitting of the feasible set into polyhedral and positive semidefinite constraints. Moreover, both relaxations can be further strengthened by several types of cutting planes.

To solve both relaxations, we use an ADMM update scheme, see~\eqref{ADMMall1}, that is incorporated in a cutting-plane framework, leading to the so-called CP--ADMM, see Algorithm~\ref{Alg:CP_ADMM}. The cutting planes are handled in the polyhedral subproblem by exploiting a semi-parallelized version of Dykstra's cyclic projection algorithm, see Section~\ref{sect:Dykstra}. The CP--ADMM benefits from warm starts whenever new cuts are added, provides valid lower bounds even after solving with low precision, and can be implemented efficiently. A particular ingredient of the CP--ADMM are the projectors onto polyhedra induced by the cutting planes. Projection operators for three types of cutting planes that are effective for the graph partition problem, i.e., the triangle, independent set and BQP inequalities, are derived in Lemma~\ref{Lem:triangle}, \ref{Lem:clique} and \ref{Lem:bqp}, respectively.

Numerical experiments show that using our CP--ADMM algorithm we are able to produce high-quality bounds for graphs up to 1,024~vertices. We experimented with several graph types from the literature. For structured graphs of medium size and the $2$-EP we often close the gap in a few seconds or at most a couple of minutes. For bisection problems on those graphs, we also close the gaps in many cases. For larger graphs, we are able to add polyhedral cuts roughly up to 50,000 for $k$-EP and 24,000 for the GBP within 2 hours, which results in strong lower bounds. Our results provide benchmarks for solving medium and large scale graph partition problems.

\medskip
This research can be extended in several directions. Motivated by the optimistic results for the $k$-EP and the GBP, we expect 
that  strong bounds from DNN relaxations with additional cutting planes can be obtained
for other graph partition problems, such as the vertex separator problem and the maximum cut problem. Moreover, since the major ingredients of our algorithm are presented generally, establishing an approach for solving general DNN relaxations with additional cutting planes is an interesting future research direction.
Our results also provide new perspectives on solving large-scale optimization problems to optimality  by using SDPs.

\subsection*{Acknowledgements}
We would like to thank William Hager for providing us instances from the paper~\cite{Hager2013AnEA}.

\begin{landscape}

	\pgfplotstableread[col
	sep=comma]{Hager_datafile1.tex}\data
	\pgfplotstabletypeset[
	font=\footnotesize,
	multicolumn names, 
	columns={graph_string,n,k,UB,round_DNN_lb,DNN_clocktime,round_cuts_lb,cuts_clocktime,imp_lb_round,TotalTriCuts,TotalCliqueCuts, iter, cut_add_occasions_total},
	fixed zerofill,
	fixed,
 	columns/graph_string/.style={string type,
		column name={graph}
	},
	columns/n/.style={int detect,
		column name={$n$},
	},  
	columns/k/.style={int detect,
		column name={$k$},
	},
	columns/UB/.style={int detect,
		column name={ub},	
	    assign column name/.style={
		/pgfplots/table/column name={\multicolumn{1}{c|}{##1}}}
	},
	columns/round_DNN_lb/.style={int detect,
	column name=lb,
	},
	columns/DNN_clocktime/.style={precision=2,
    column name=clocktime,
		assign column name/.style={
		/pgfplots/table/column name={\multicolumn{1}{c|}{##1}}}
	},
	columns/round_cuts_lb/.style={int detect,
	column name=lb,
	},
	columns/cuts_clocktime/.style={precision=2,
	column name=clocktime,
	},
	columns/imp_lb_round/.style={preproc/expr={100*##1},
	column name=imp.,
	},
	columns/TotalTriCuts/.style={int detect,
	column name=\# tri. cuts,
	},
	columns/TotalCliqueCuts/.style={int detect,
	column name=\# ind. set cuts,
 	},
 	columns/iter/.style={int detect,
	column name=\# iter.,
	},
 	columns/cut_add_occasions_total/.style={int detect,
	column name=\# outer loops,
	},
	columns addvline/.style={
	every col no #1/.style={
	column type/.add={}{|}}
    },
	columns addvline/.list={3,5},
	    column type=r,
	every head row/.style={
	before row={
	\caption{Instances considered in~\cite[Table~9]{Hager2013AnEA}, $k=2$} \label{tab:hager_table9_tol=10-3_mixcuts=1}\\
	\toprule
	&  &  &  & \multicolumn{2}{c|}{DNN}& \multicolumn{7}{c}{DNN+cuts}\\
	},
	after row={
	&  &  &  & (rounded) &  \multicolumn{1}{c|}{(\si{\second})} & (rounded) & \multicolumn{1}{c}{(\si{\second})}  & \multicolumn{1}{c}{\%} &  &  & & 
	\endfirsthead
	\midrule}
	},
	every last row/.style={after row=\bottomrule}
	]{\data} 

	\pgfplotstableread[col sep=comma]{largeGraphSparse5n_datafile3.tex}\data
	\pgfplotstabletypeset[
	font=\footnotesize,
	multicolumn names, 
	columns={graph_string,n,k,UB,round_DNN_lb,DNN_clocktime,round_cuts_lb,cuts_clocktime,imp_lb_round,TotalTriCuts,TotalCliqueCuts, iter, cut_add_occasions_total},
	fixed zerofill,
	fixed,
 	columns/graph_string/.style={string type,
		column name={graph},
	},
	columns/n/.style={int detect,
		column name={$n$},
	},  
	columns/k/.style={int detect,
		column name={$k$},
	},
	columns/UB/.style={int detect,
		column name={ub},	
	    assign column name/.style={
		/pgfplots/table/column name={\multicolumn{1}{c|}{##1}}}
	},
	columns/round_DNN_lb/.style={int detect,
	column name=lb,
	},
	columns/DNN_clocktime/.style={precision=2, 
    column name=clocktime,
		assign column name/.style={
		/pgfplots/table/column name={\multicolumn{1}{c|}{##1}}}
	},
	columns/round_cuts_lb/.style={int detect,
	column name=lb,
	},
	columns/cuts_clocktime/.style={precision=2,
	column name=clocktime,
	},
	columns/imp_lb_round/.style={preproc/expr={100*##1}, 
	column name=imp. ,
	},
	columns/TotalTriCuts/.style={int detect,
	column name=\# tri. cuts,
	},
	columns/TotalCliqueCuts/.style={int detect,
	column name=\# ind. set cuts,
 	},
 	columns/iter/.style={int detect,
	column name=\# iter.,
	},
 	columns/cut_add_occasions_total/.style={int detect,
	column name=\# outer loops,
	},
	columns addvline/.style={
	every col no #1/.style={
	column type/.add={}{|}}
    },
	columns addvline/.list={3,5},
	    column type=r,
	every head row/.style={
	before row={
	\caption{Large $G$ and $U$ graphs from~\cite{johnson1989optimization}, $k=2$}\label{tab:LargeGraph_GU_5n_mixcuts=1_tol10-3}\\
	\toprule
	&  &  &  & \multicolumn{2}{c|}{DNN}& \multicolumn{7}{c}{DNN+Cuts}\\
	},
	after row={
	&  &  &  &  (rounded) &\multicolumn{1}{c|}{(\si{\second})}   & (rounded)&  \multicolumn{1}{c}{(\si{\second})} & \multicolumn{1}{c}{\%} &  &  & &
	\endfirsthead
	\midrule}
	},
	every last row/.style={after row=\bottomrule}
	]{\data} 

	\pgfplotstableread[col sep=comma]{Hager_datafile2.tex}\data
	
	\pgfplotstabletypeset[
	font=\footnotesize,
	multicolumn names, 
	columns={graph_string,n,k,UB,round_DNN_lb,DNN_clocktime,round_cuts_lb,cuts_clocktime,imp_lb_round,TotalTriCuts,TotalCliqueCuts, iter, cut_add_occasions_total},
	fixed zerofill,
	fixed,
	skip rows between index={0}{1},
	skip rows between index={7}{8},
 	columns/graph_string/.style={string type,
		column name={graph},
	},
	columns/n/.style={int detect,
		column name={$n$},
	},  
	columns/k/.style={int detect,
		column name={$k$},
	},
	columns/UB/.style={int detect,
		column name={ub},	
	    assign column name/.style={
		/pgfplots/table/column name={\multicolumn{1}{c|}{##1}}}
	},
	columns/round_DNN_lb/.style={int detect,
	column name={lb},
	},
	columns/DNN_clocktime/.style={precision=2,
    column name=clocktime,
		assign column name/.style={
		/pgfplots/table/column name={\multicolumn{1}{c|}{##1}}}
	},
	columns/round_cuts_lb/.style={int detect,
	column name=lb,
	},
	columns/cuts_clocktime/.style={precision=2,
	column name=clocktime,
	},
	columns/imp_lb_round/.style={preproc/expr={100*##1},
	column name=imp.,
	},
	columns/TotalTriCuts/.style={int detect,
	column name=\# tri. cuts,
	},
	columns/TotalCliqueCuts/.style={int detect,
	column name=\# ind. set cuts,
 	},
 	columns/iter/.style={int detect,
	column name=\# iter.,
	},
 	columns/cut_add_occasions_total/.style={int detect,
	column name=\# outer loops,
	},
	columns addvline/.style={
	every col no #1/.style={
	column type/.add={}{|}}
    },
	columns addvline/.list={3,5},
	    column type=r,
    every row no 2/.style={before row={\midrule}},
    every row no 7/.style={before row={\midrule}},
    every row no 13/.style={before row={\midrule}},
	every head row/.style={
	before row={
	\caption{Instances considered in~\cite{HelmbergEtAl}, $k \in \{3,4,5,6 \}$}\label{tab:hager_table9DiffK_tol=10-3_mixcuts=1}\\
	\toprule
	&  &  &  & \multicolumn{2}{c|}{DNN}& \multicolumn{7}{c}{DNN+cuts}\\
	},
	after row={
	&  &  &  & (rounded) & \multicolumn{1}{c|}{(\si{\second})} & (rounded)  & \multicolumn{1}{c}{(\si{\second})} &\multicolumn{1}{c}{\%} &  &  & &
	\endfirsthead
	\midrule}},
	every last row/.style={after row=\bottomrule}
	]{\data} 


	\pgfplotstableread[col sep=comma]{largeGraphSparseSpinglass2pm_datafile1.tex}\data
	\pgfplotstabletypeset[
	font=\footnotesize,
	multicolumn names, 
	columns={graph,n,k,UB,round_DNN_lb,DNN_clocktime,round_cuts_lb,cuts_clocktime,imp_lb_round,TotalTriCuts,TotalCliqueCuts, iter, cut_add_occasions_total},
	fixed zerofill,
	fixed,
 	columns/graph/.style={string type,
		column name={graph}
	},
	columns/n/.style={int detect,
		column name={$n$},
	},  
	columns/k/.style={int detect,
		column name={$k$},
	},
	columns/UB/.style={int detect,
		column name={ub},	
	    assign column name/.style={
		/pgfplots/table/column name={\multicolumn{1}{c|}{##1}}}
	},
	columns/round_DNN_lb/.style={int detect,
	column name=lb,
	},
	columns/DNN_clocktime/.style={precision=2,
    column name=clocktime,
		assign column name/.style={
		/pgfplots/table/column name={\multicolumn{1}{c|}{##1}}}
	},
	columns/round_cuts_lb/.style={int detect,
	column name=lb,
	},
	columns/cuts_clocktime/.style={precision=2,
	column name=clocktime,
	},
	columns/imp_lb_round/.style={preproc/expr={100*##1},
	column name=imp.,
	},
	columns/TotalTriCuts/.style={int detect,
	column name=\# tri. cuts,
	},
	columns/TotalCliqueCuts/.style={int detect,
	column name=\# ind. set cuts,
 	},
 	columns/iter/.style={int detect,
	column name=\# iter.,
	},
 	columns/cut_add_occasions_total/.style={int detect,
	column name=\# outer loops,
	},
	columns addvline/.style={
	every col no #1/.style={
	column type/.add={}{|}}
    },
	columns addvline/.list={3,5},
	    column type=r,
	every row no 5/.style={before row={\midrule}},
    every row no 12/.style={before row={\midrule}},
	every head row/.style={
	before row={
	\caption{Two-dimensional spinglass graphs, $k\in\{3,4,5\}$}\\
	\toprule
	&  &  &  & \multicolumn{2}{c|}{DNN}& \multicolumn{7}{c}{DNN+cuts}\\
	},
	after row={
	&  &  &  & (rounded) & \multicolumn{1}{c|}{(\si{\second})} &(rounded)  & \multicolumn{1}{c}{(\si{\second})}  &  \multicolumn{1}{c}{\%} &  &  & &
	\endfirsthead
	\midrule}
	},
	every last row/.style={after row=\bottomrule}
	]{\data} 

	\pgfplotstableread[col sep=comma]{largeGraphSparseSpinglass3pm_datafile1.tex}\data
	\pgfplotstabletypeset[
	font=\footnotesize,
	multicolumn names, 
	columns={graph,n,k,UB,round_DNN_lb,DNN_clocktime,round_cuts_lb,cuts_clocktime,imp_lb_round,TotalTriCuts,TotalCliqueCuts, iter, cut_add_occasions_total},
	fixed zerofill,
	fixed,
 	columns/graph/.style={string type,
		column name={graph}
	},
	columns/n/.style={int detect,
		column name={$n$},
	},  
	columns/k/.style={int detect,
		column name={$k$},
	},
	columns/UB/.style={int detect,
		column name={ub},	
	    assign column name/.style={
		/pgfplots/table/column name={\multicolumn{1}{c|}{##1}}}
	},
	columns/round_DNN_lb/.style={int detect,
	column name=lb,
	},
	columns/DNN_clocktime/.style={precision=2,
    column name=clocktime,
		assign column name/.style={
		/pgfplots/table/column name={\multicolumn{1}{c|}{##1}}}
	},
	columns/round_cuts_lb/.style={int detect,
	column name=lb,
	},
	columns/cuts_clocktime/.style={precision=2,
	column name=clocktime,
	},
	columns/imp_lb_round/.style={preproc/expr={100*##1},
	column name=imp.,
	},
	columns/TotalTriCuts/.style={int detect,
	column name=\# tri. cuts,
	},
	columns/TotalCliqueCuts/.style={int detect,
	column name=\# ind. set cuts,
 	},
 	columns/iter/.style={int detect,
	column name=\# iter.,
	},
 	columns/cut_add_occasions_total/.style={int detect,
	column name=\# outer loops,
	},
	columns addvline/.style={
	every col no #1/.style={
	column type/.add={}{|}}
    },
    column type=r,
	columns addvline/.list={3,5},
	every row no 1/.style={before row={\midrule}},
    every row no 3/.style={before row={\midrule}},
	every head row/.style={
	before row={
	\caption{Three-dimensional spinglass graphs, $k\in\{3,4,5\}$}\\
	\toprule
	&  &  &  & \multicolumn{2}{c|}{DNN}& \multicolumn{7}{c}{DNN+cuts}\\
	},
	after row={
	&  &  &  & (rounded) & \multicolumn{1}{c|}{(\si{\second})} & (rounded)  & \multicolumn{1}{c}{(\si{\second})} & \multicolumn{1}{c}{\%} &  &  & &
	\endfirsthead
	\midrule}
	},
	every last row/.style={after row=\bottomrule}
	]{\data} 

	\pgfplotstableread[col sep=comma]{largeGraphSparse5n_datafile1.tex}\data
	
	\pgfplotstabletypeset[
	font=\footnotesize,
	multicolumn names, 
	columns={graph_string,n,k,UB,round_DNN_lb,DNN_clocktime,round_cuts_lb,cuts_clocktime,imp_lb_round,TotalTriCuts,TotalCliqueCuts, iter, cut_add_occasions_total},
	fixed zerofill,
	fixed,
 	columns/graph_string/.style={string type,
		column name={graph},
	},
	columns/n/.style={int detect,
		column name={$n$},
	},  
	columns/k/.style={int detect,
		column name={$k$},
	},
	columns/UB/.style={int detect,
		column name={ub},	
	    assign column name/.style={
		/pgfplots/table/column name={\multicolumn{1}{c|}{##1}}}
	},
	columns/round_DNN_lb/.style={int detect,
	column name=lb,
	},
	columns/DNN_clocktime/.style={precision=2,
    column name=clocktime,
		assign column name/.style={
		/pgfplots/table/column name={\multicolumn{1}{c|}{##1}}}
	},
	columns/round_cuts_lb/.style={int detect,
	column name=lb,
	},
	columns/cuts_clocktime/.style={precision=2,
	column name=clocktime,
	},
	columns/imp_lb_round/.style={preproc/expr={100*##1},
	column name=imp.,
	},
	columns/TotalTriCuts/.style={int detect,
	column name=\# tri. cuts,
	},
	columns/TotalCliqueCuts/.style={int detect,
	column name=\# ind. set cuts,
 	},
 	columns/iter/.style={int detect,
	column name=\# iter.,
	},
 	columns/cut_add_occasions_total/.style={int detect,
	column name=\# outer loops,
	},
	columns addvline/.style={
	every col no #1/.style={
	column type/.add={}{|}}
    },
	columns addvline/.list={3,5},
	    column type=r,
	every head row/.style={
	before row={
	\caption{Large $G$ and $U$ graphs from~\cite{johnson1989optimization}, $k=4$} \label{tab:LargeGraph_GU_5n_k=4_mixcuts=1_tol=10-3}\\
	\toprule
	&  &  &  & \multicolumn{2}{c|}{DNN}& \multicolumn{7}{c}{DNN+cuts}\\
	},
	after row={
	&  &  &  &  (rounded) & \multicolumn{1}{c|}{(\si{\second})}   & (rounded) & \multicolumn{1}{c}{(\si{\second})}  & \multicolumn{1}{c}{\%} &  &  & &
	\endfirsthead
	\midrule}
	},
	every last row/.style={after row=\bottomrule}
	]{\data} 

	\pgfplotstableread[col sep=comma]{largeGraphSparse5n_datafile2.tex}\data
	\pgfplotstabletypeset[
	font=\footnotesize,
	multicolumn names, 
	columns={graph_string,n,k,UB,round_DNN_lb,DNN_clocktime,round_cuts_lb,cuts_clocktime,imp_lb_round,TotalTriCuts,TotalCliqueCuts, iter, cut_add_occasions_total},
	fixed zerofill,
	fixed,
 	columns/graph_string/.style={string type,
		column name={graph},
	},
	columns/n/.style={int detect,
		column name={$n$},
	},  
	columns/k/.style={int detect,
		column name={$k$},
	},
	columns/UB/.style={int detect,
		column name={ub},	
	    assign column name/.style={
		/pgfplots/table/column name={\multicolumn{1}{c|}{##1}}}
	},
	columns/round_DNN_lb/.style={int detect,
	column name=lb,
	},
	columns/DNN_clocktime/.style={precision=2,
    column name=clocktime,
		assign column name/.style={
		/pgfplots/table/column name={\multicolumn{1}{c|}{##1}}}
	},
	columns/round_cuts_lb/.style={int detect,
	column name=lb,
	},
	columns/cuts_clocktime/.style={precision=2, 
	column name=clocktime,
	},
	columns/imp_lb_round/.style={preproc/expr={100*##1},
	column name=imp., 
	},
	columns/TotalTriCuts/.style={int detect,
	column name=\# tri. cuts,
	},
	columns/TotalCliqueCuts/.style={int detect,
	column name=\# ind. set cuts,
 	},
 	columns/iter/.style={int detect,
	column name=\# iter.,
	},
 	columns/cut_add_occasions_total/.style={int detect,
	column name=\# outer loops,
	},
	columns addvline/.style={
	every col no #1/.style={
	column type/.add={}{|}}
    },
	columns addvline/.list={3,5},
	    column type=r,
	every head row/.style={
	before row={
	\caption{Large $G$ and $U$ graphs from~\cite{johnson1989optimization}, $k=5$} \label{tab:LargeGraph_GU_5n_k=5_mixcuts=1_tol=10-3}\\
	\toprule
	&  &  &  & \multicolumn{2}{c|}{DNN}& \multicolumn{7}{c}{DNN+cuts}\\
	},
	after row={
	&  &  &  &  (rouneded) &\multicolumn{1}{c|}{(\si{\second})}   & (rouneded)& \multicolumn{1}{c}{(\si{\second})}  & \multicolumn{1}{c}{\%}&  &  & &
	\endfirsthead
	\midrule}
	},
	every last row/.style={after row=\bottomrule}
	]{\data} 

	\pgfplotstableread[col sep=comma]{hagerBisection_datafile1.tex}\data
	\pgfplotstabletypeset[
	font=\footnotesize,
	multicolumn names, 
	columns={graph_string,n,m1,UB,round_DNN_lb,DNN_clocktime,round_cuts_lb,cuts_clocktime,imp_lb_round,TotalBQPTriCuts, iter, cut_add_occasions_total},
	fixed zerofill,
	fixed,
 	columns/graph_string/.style={string type,
		column name={graph}
	},
	columns/n/.style={int detect,
		column name={$n$},
	},  
	columns/m1/.style={int detect,
		column name={$m_1$},
	},
	columns/UB/.style={int detect,
		column name={ub},	
	    assign column name/.style={
		/pgfplots/table/column name={\multicolumn{1}{c|}{##1}}}
	},
	columns/round_DNN_lb/.style={int detect,
	column name=lb,
	},
	columns/DNN_clocktime/.style={precision=2, 
    column name=clocktime,
		assign column name/.style={
		/pgfplots/table/column name={\multicolumn{1}{c|}{##1}}}
	},
	columns/round_cuts_lb/.style={int detect,
	column name=lb,
	},
	columns/cuts_clocktime/.style={precision=2,
	column name=clocktime,
	},
	columns/imp_lb_round/.style={preproc/expr={100*##1},
	column name=imp.,
	},
	columns/TotalBQPTriCuts/.style={int detect,
	column name=\# BQP cuts,
	},
 	columns/iter/.style={int detect,
	column name=\# iter.,
	},
 	columns/cut_add_occasions_total/.style={int detect,
	column name=\# outer loops,
	},
	columns addvline/.style={
	every col no #1/.style={
	column type/.add={}{|}}
    },
	columns addvline/.list={3,5},
	column type=r,
	every head row/.style={
	before row={
	\caption{Instances considered in~\cite[Table~9]{Hager2013AnEA}, $m = (m_1,n-m_1)$}\label{tab:TableHager_finaltol=0.001_bisection} \\
	\toprule
	&  &  &  & \multicolumn{2}{c|}{DNN}& \multicolumn{6}{c}{DNN+cuts}\\
	},
	after row={
	&  &  &  & (rounded) &  \multicolumn{1}{c|}{(\si{\second})} & (rounded) & \multicolumn{1}{c}{(\si{\second})}& \multicolumn{1}{c}{\%} &  &  &  
	\endfirsthead
	\midrule}
	},
	every last row/.style={after row=\bottomrule}
	]{\data} 

	\pgfplotstableread[col sep=comma]{spinglass2pmBisection_datafile1.tex}\data
	\pgfplotstabletypeset[
	font=\footnotesize,
	multicolumn names, 
	columns={graph_string,n,m1,UB,round_DNN_lb,DNN_clocktime,round_cuts_lb,cuts_clocktime,imp_lb_round,TotalBQPTriCuts, iter, cut_add_occasions_total},
	fixed zerofill,
	fixed,
 	columns/graph_string/.style={string type,
		column name={graph}
	},
	columns/n/.style={int detect,
		column name={$n$},
	},  
	columns/m1/.style={int detect,
		column name={$m_1$},
	},
	columns/UB/.style={int detect,
		column name={ub},	
	    assign column name/.style={
		/pgfplots/table/column name={\multicolumn{1}{c|}{##1}}}
	},
	columns/round_DNN_lb/.style={int detect,
	column name=lb,
	},
	columns/DNN_clocktime/.style={precision=2,
    column name=clocktime,
		assign column name/.style={
		/pgfplots/table/column name={\multicolumn{1}{c|}{##1}}}
	},
	columns/round_cuts_lb/.style={int detect,
	column name=lb,
	},
	columns/cuts_clocktime/.style={precision=2,
	column name=clocktime,
	},
	columns/imp_lb_round/.style={preproc/expr={100*##1},
	column name=imp.,
	},
	columns/TotalBQPTriCuts/.style={int detect,
	column name=\# BQP cuts,
	},
 	columns/iter/.style={int detect,
	column name=\# iter.,
	},
 	columns/cut_add_occasions_total/.style={int detect,
	column name=\# outer loops,
	},
	columns addvline/.style={
	every col no #1/.style={
	column type/.add={}{|}}
    },
	columns addvline/.list={3,5},
	column type=r,
	every head row/.style={
	before row={
	\caption{Large $G$ graphs from~\cite{johnson1989optimization} and two-dimensional spinglass graphs, $m = (m_1,n-m_1)$}\label{tab:Table_largeG_Spinglass2pm_bisection} \\
	\toprule
	&  &  &  & \multicolumn{2}{c|}{DNN}& \multicolumn{6}{c}{DNN+cuts}\\
	},
	after row={
	&  &  &  & (rounded) &  \multicolumn{1}{c|}{(\si{\second})} & (rounded) & \multicolumn{1}{c}{(\si{\second})} & \multicolumn{1}{c}{\%} &  &  &  
	\endfirsthead
	\midrule}
	},
	every last row/.style={after row=\bottomrule}
	]{\data} 

\end{landscape}

\clearpage
\bibliographystyle{plainnat}
\bibliography{mybib} 


\appendix

\section{Projection onto polyhedral sets} \label{App:projectorX}

One of the ingredients of the CP--ADMM is the orthogonal projection onto a polyhedral set. More precisely, the $X$-subproblem in~\eqref{X_sub} involves a projection onto the set $\mathcal{X}$, where $\mathcal{X}$ is given in~\eqref{setX} and \eqref{setXGB} for the $k$-EP and GBP, respectively. 
We focus here on the projector for the GBP. The projector for the $k$-EP, which has a simpler structure, can be obtained similarly.

Recall from~\eqref{setXGB} that the polyhedral set $\mathcal{X}_{BP}$ looks as follows: 
\begin{align*}
\mathcal{X}_{BP} &= \left \{  
X = \begin{pmatrix}
x^0 & (x^1)^\top & (x^2)^\top \\
x^1 & X^{11} & X^{12} \\
x^2 & X^{21} & X^{22}
\end{pmatrix} \in  {\mathcal S}^{2n+1} : ~~ \begin{aligned} & {\mathcal G}_{\mathcal J}({X}) = \mathbf{0}, ~{X}_{1,1}=1,~\trace(X^{ii})=m_i,~ i\in [2], \\
& \diag(X^{11})+\diag(X^{22})={\mathbf 1}_n,~
X {\mathbf u}_1 =\diag(X),\\
& \mathbf{0} \leq {X} \leq {\mathbf J} 
\end{aligned}  \right\}.
\end{align*}
Let $\mathcal{P}_{\mathcal{X}_{BP}} : \mathcal{S}^{2n+1} \to \mathcal{S}^{2n+1}$ denote the projection onto $\mathcal{X}_{BP}$.

Observe that each constraint that defines $\mathcal{X}_{BP}$ either acts on the diagonal, first row, and first column of the matrix, or on the remaining entries. In the latter case, an entry $X_{ij}$ is either bounded by 0 and 1 or equals 0 if $(i,j) \in \mathcal{J}$. These projections are very simple and are given by the operators $T_{\inner}$ and $T_{\text{box}}$ in Table~\ref{TableOperators}.  

Next, we focus on the entries on the diagonal, first row\textcolor{black}{,} and first column of the orthogonal projection. Suppose $Y = \mathcal{P}_{\mathcal{X}_{BP}}(X)$ 
and let $y_1 := \diag(Y^{11})$ and $y_2 := \diag(Y^{22})$. Then $y_1$ and $y_2$ can be obtained via the following optimization problem:
\begin{align} \label{eq:projectionDiag}
    \begin{aligned} \min_{y_1,y_2 \in \mathbb{R}^n}  \quad & (y_1 - \diag(X^{11}))^\top (y_1 - \diag(X^{11})) +2(y_1 - x^1)^\top (y_1 - x^1) \\
    & + (y_2 - \diag(X^{22}))^\top (y_2 - \diag(X^{22})) +2(y_2 - x^2)^\top (y_2 - x^2) \\
    \text{s.t.} \quad & \mathbf{1}_n^\top y_1 = m_1,~ \mathbf{1}_n^\top y_2 = m_2,~ y_1 + y_2 = \mathbf{1}_n,~ y_1 \geq \mathbf{0}_n,~  y_2 \geq \mathbf{0}_n. \end{aligned}
\end{align}
Using basic algebra, one can show that the optimal $y_1$ to~\eqref{eq:projectionDiag} is attained by the minimizer of the following optimization problem: 
\begin{align} \label{eq:projectionDiag_rewritten}
 \begin{aligned}   \min_{y_1 \in \mathbb{R}^n}  \quad & \left\Vert  y_1 - \left( \frac{1}{6} (\diag(X^{11}) - \diag(X^{22})) + \frac{1}{3}(x^1 - x^2) + \frac{1}{2}\mathbf{1}_n\right)\right\Vert_2^2 \\
    \text{s.t.} \quad & \mathbf{1}_n^\top y_1 = m_1,~ \mathbf{0}_n \leq y_1 \leq \mathbf{1}_n, \end{aligned}
\end{align}
while the corresponding optimal $y_2$ to~\eqref{eq:projectionDiag} is $y_2 = \bold{1}_n - y_1$. Observe that \eqref{eq:projectionDiag_rewritten} is equivalent to a projection onto the capped simplex $\bar{\Delta}(m_1) = \{ y \in \mathbb{R}^n \, : \, \, \mathbf{1}_n^\top y = m_1,~ \mathbf{0}_n \leq y \leq \mathbf{1}_n\}$. The projection onto $\bar{\Delta}(m_1)$ we denote by $\mathcal{P}_{\bar{\Delta}(m_1)} : \mathbb{R}^n \to \mathbb{R}^n$, which can be performed efficiently, see~\cite{AngEtAl}. We define the operator $T_{\arrow}$, see Table~\ref{TableOperators}, to embed the optimal $y_1$ and $y_2$ in the space $\mathcal{S}^{2n+1}$. 
\begin{table}[H]
\centering
\footnotesize
\begin{tabular}{@{}ccll@{}}
\toprule
     \multicolumn{3}{c}{Operator}                                               & Description                                                                                                                                                 \\ \midrule
$T_{\inner}$     & : & $\mathcal{S}^{2n+1} \to \mathcal{S}^{2n+1}$                 & \begin{tabular}[c]{@{}l@{}}
$T_{\inner}\left( X \right)_{ij} =  0$ if $i = 1$ or $j = 1$ or $i = j$ or $(i,j) \in \mathcal{J}$ \\
and $T_{\inner}(X)_{ij} = X_{ij}$ otherwise.
\end{tabular}\\ [0.5cm]
$T_{\text{box}}$ & : & $\mathcal{S}^{2n+1} \to \mathcal{S}^{2n+1}$ & $T_{\text{box}}(X)_{ij} = \min( \max( X_{ij}, 0), 1)$ for all $(i,j)$. \\[0.3cm]
$T_\arrow$ & : & $\mathbb{R}^n \to \mathcal{S}^{2n+1}$ & $T_\arrow(x) = \begin{pmatrix}
1 & x^\top & (\mathbf{1}_n - x)^\top \\ x & \Diag(x) & \mathbf{0} \\
\mathbf{1}_n - x & \mathbf{0} & \Diag(\mathbf{1}_n - x)
\end{pmatrix}$. \\
\bottomrule
\end{tabular}
\caption{Overview of operators and their definitions. \label{TableOperators}}
\end{table}
\noindent Now, the projector $\mathcal{P}_{\mathcal{X}_{BP}}$ can be written out explicitly as follows: 

\begin{minipage}{0.95\textwidth}
\begin{flalign*}
    \mathcal{P}_{\mathcal{X}_{BP}}\left( \begin{pmatrix}
    x^0 & (x^1)^\top & (x^2)^\top \\
    x^1 & X^{11} & X^{12} \\
    x^2 & X^{21} & X^{22}
    \end{pmatrix} \right)  = T_{\text{box}} \left( T_{\inner} \left(  \begin{pmatrix}
    x^0 & (x^1)^\top & (x^2)^\top \\
    x^1 & X^{11} & X^{12} \\
    x^2 & X^{21} & X^{22}
    \end{pmatrix} \right) \right) && \end{flalign*}
    \begin{flalign*}
     && + T_{\arrow}\left( \mathcal{P}_{\bar{\Delta}(m_1)}\left( \frac{1}{6} (\diag(X^{11}) - \diag(X^{22})) + \frac{1}{3}(x^1 - x^2) + \frac{1}{2}\mathbf{1}_n\right) \right). 
\end{flalign*}
\end{minipage}

\section{Proofs for Cutting Plane Projectors}
\subsection{Proof of Lemma \ref{Lem:triangle}} \label{App:prooftriangle}
The first statement is trivial. Now assume $M \notin \mathcal{H}^\Delta_{ijl}$, i.e., $M_{ij} + M_{il} > \frac{k-1}{k} + M_{jl}$. The projection of $M$ onto $\mathcal{H}^\Delta_{ijl}$ is the solution to $\min_{\hat{M} \in \mathcal{S}^n}\left\{ \Vert \hat{M} - M \Vert_F^2 \, : \, \, \hat{M} \in \mathcal{H}^\Delta_{ijl}\right\}$. Since the inequality that describes $\mathcal{H}^\Delta_{ijl}$ only involves the submatrix induced by indices $i, j$ and $l$, we can restrict ourselves to the following convex optimization problem: 
\begin{align*}
\min_{\alpha, \beta, \gamma} \quad & 2(\alpha - M_{ij})^2 + 2(\beta - M_{il})^2 + 2(\gamma - M_{jl})^2 \\
\text{s.t.} \quad & \alpha + \beta \leq \frac{k-1}{k} + \gamma. 
\end{align*}
Let $\lambda \geq 0$ denote the Lagrange multiplier for the inequality, then the KKT conditions imply the following system: 
\begin{align*}
\begin{cases}
4(\alpha - M_{ij}) + \lambda = 0 \\
4(\beta - M_{il}) + \lambda = 0 \\
4(\gamma - M_{jl}) - \lambda = 0 \\
\lambda (\alpha + \beta - \frac{k-1}{k} - \gamma) = 0 \\
\alpha + \beta \leq \frac{k-1}{k} + \gamma \\
\lambda \geq 0. 
\end{cases}
\end{align*}
By complementarity, we have either $\lambda = 0$ or $\alpha + \beta - \frac{k-1}{k} - \gamma = 0$. The first case leads to the solution $\hat{M} = M$ that is a KKT-point if $M \in \mathcal{H}^\Delta_{ijl}$. If $M \notin \mathcal{H}^\Delta_{ijl}$, then $\alpha + \beta - \frac{k-1}{k} - \gamma = 0$. It follows from the first three conditions of the above system that we have
\begin{align*}
\alpha = M_{ij} - \frac{1}{4}\lambda, \quad \beta = M_{il} - \frac{1}{4} \lambda,  \quad \gamma = M_{jl} + \frac{1}{4}\lambda. 
\end{align*}
Substitution into $\alpha + \beta - \frac{k-1}{k} - \gamma = 0$ yields
\begin{align*}
0 & = M_{ij} - \frac{1}{4}\lambda + M_{il} - \frac{1}{4}\lambda - \frac{k-1}{k} - M_{jl} - \frac{1}{4}\lambda  
\\
\Longleftrightarrow \quad \lambda & = \frac{4}{3}\left( M_{ij} + M_{il} - M_{jl} - \frac{k-1}{k} \right). 
\end{align*}
Using this expression of the Lagrange multiplier gives the optimal values for $\alpha, \beta$ and $\gamma$:
\begin{align*}
\alpha & = M_{ij} - \frac{1}{4}\cdot\frac{4}{3}\left( M_{ij} + M_{il} - M_{jl} - \frac{k-1}{k} \right) = \frac{2}{3} M_{ij} - \frac{1}{3} M_{il} + \frac{1}{3} M_{jl} + \frac{1}{3} - \frac{1}{3k}, \\
\beta & = M_{il} - \frac{1}{4}\cdot \frac{4}{3}\left( M_{ij} + M_{il} - M_{jl} - \frac{k-1}{k} \right) = -\frac{1}{3}M_{ij} + \frac{2}{3} M_{il} + \frac{1}{3} M_{jl} + \frac{1}{3} - \frac{1}{3k}, \\
\gamma & = M_{jl} + \frac{1}{4} \cdot \frac{4}{3}\left( M_{ij} + M_{il} - M_{jl} - \frac{k-1}{k} \right) = \frac{1}{3}M_{ij} + \frac{1}{3}M_{il} + \frac{2}{3} M_{jl} - \frac{1}{3} + \frac{1}{3k}.  
\end{align*}
Setting $\hat{M}_{ij} = \hat{M}_{ji}= \alpha, \hat{M}_{il} = \hat{M}_{li} =\beta$ and $\hat{M}_{jl} = \hat{M}_{lj} = \gamma$ gives the desired result. \qedsymbol

\subsection{Proof of Lemma \ref{Lem:clique}} \label{App:proofclique}
We use a similar technique as used in the proof of Lemma \ref{Lem:triangle}. The statement is trivial if $M \in \mathcal{H}^{IS}_I$. If $M \notin \mathcal{H}^{IS}_I$, it suffices to consider the following optimization problem: 
\begin{align*}
\min_{z} \quad &  \sum_{i, j \in I, i < j} 2\left( z_{ij} - M_{ij} \right)^2 \\
\text{s.t.} \quad & \sum_{i, j \in I, i < j} z_{ij} \geq \frac{1 - k}{2}. 
\end{align*}
Since this problem is convex, we restrict ourselves to solutions satisfying the KKT conditions. Let $\lambda \geq 0$ denote the Lagrange multiplier for the inequality. The KKT conditions read as follows: 
\begin{align*}
\begin{cases} 4(z_{ij} - M_{ij} ) - \lambda = 0 \qquad \forall i, j \in I, ~i < j  \\
\lambda \left( \frac{1 - k}{2} - \sum_{i, j \in I, i < j} z_{ij} \right) = 0 \\
\sum_{i,j \in I, i < j} z_{ij} \geq \frac{1 - k}{2} \\
\lambda \geq 0.
\end{cases}
\end{align*}
Complementarity implies that either $\lambda = 0$ or $\frac{1 - k}{2} - \sum_{i, j \in I, i < j} z_{ij} = 0$. The former case leads to the solution $z_{ij} = M_{ij}$ for all $i,j \in I, i < j$, which is only a KKT point if $M \in \mathcal{H}^{IS}_I$. Since this is not the case, we have $\sum_{i, j \in I, i < j} z_{ij} = \frac{1 - k}{2}$. It follows from the system's first equation that $z_{ij} = M_{ij} + \frac{1}{4}\lambda$ for all $i, j \in I, i < j$. Substitution into $\sum_{i, j \in I, i < j} z_{ij} = \frac{1 - k}{2}$ yields: 
\begin{align*}
\sum_{i,j\in I, i < j} \left( M_{ij} + \frac{1}{4}\lambda \right) = \frac{1 - k}{2} \quad & \Longleftrightarrow \quad  \binom{k+1}{2} \frac{1}{4} \lambda + \sum_{i,j \in I, i < j} M_{ij} = \frac{1 - k}{2} \\
& \Longleftrightarrow \quad \frac{k(k+1)}{2 \cdot 4} \lambda = \frac{1 - k}{2} - \sum_{i,j \in I, i < j} M_{ij} \\
& \Longleftrightarrow \quad \lambda = \frac{8}{k(k+1)} \left( \frac{1 - k}{2} - \sum_{i,j \in I, i < j} M_{ij} \right). 
\end{align*}
Hence, this implies that for all $p,q \in I, p < q$ we have
\begin{align*}
z_{pq} = M_{pq} + \frac{1}{4}\lambda = M_{pq} - \frac{k-1}{k(k+1)} - \frac{2}{k(k+1)} \sum_{i,j \in I, i < j}M_{ij}.
\end{align*}
Setting $\hat{M}_{pq} = z_{pq}$ for all $p, q \in I, p  < q$, $\hat{M}_{pq} = z_{qp}$ for all $p,q \in I, p > q$, and $\hat{M}_{pq} = M_{pq}$ otherwise leads to the desired result. \qedsymbol

\subsection{Proof of Lemma \ref{Lem:bqp}} \label{App:proofbqp}
The projection of $M$ onto $\mathcal{H}_{ijl}^{BQP}$ is the solution to $\min_{\hat{M} \in \mathcal{S}^{2n+1}}\left\{ \Vert \hat{M} - M \Vert_F^2 \, : \, \, \hat{M} \in \mathcal{H}_{ijl}\right\}$. The inequality describing $\mathcal{H}_{ijl}^{BQP}$ only involves the pairs $(i,l), (j,l), (i,j)$ and $(l,l)$. Since $\diag(X^{11}) = x^1, \diag(X^{22}) = x^2$ and $x^1 + x^2 = {\mathbf 1}_n$ should be satisfied for $\hat{M}$, any change in $(l,l)$ also has an effect on the pairs $(1,l)$, $(l^*, l^*)$ and $(1,l^*)$, where $l^*$ is the index corresponding to $l$ in the diagonal block not containing $l$. Taking these pairs into account, we can restrict ourselves to the following convex optimization problem: 
\begin{align*}
\min_{\alpha, \beta, \gamma, \mu} \quad & \begin{aligned}[t] & 2(\alpha - M_{il})^2 + 2(\beta - M_{jk})^2 + 2(\gamma - M_{ij})^2 + (\mu - M_{ll})^2 + 2(\mu - M_{1l})^2 \\ &\quad  + (1 - \mu - M_{l^*l^*})^2 + 2(1 - \mu - M_{1l^*})^2 \end{aligned} \\
\text{s.t.} \quad & \alpha + \beta \leq  \gamma + \mu. 
\end{align*}
Let $\lambda \geq 0$ denote the Lagrange multiplier for the inequality, then the KKT conditions imply the following system: 
\begin{align*}
\begin{cases}
4(\alpha - M_{il}) + \lambda = 0 \\
4(\beta - M_{jl}) + \lambda = 0 \\
4(\gamma - M_{ij}) - \lambda = 0 \\
2(\mu - M_{ll}) + 4(\mu - M_{1l}) +2 (\mu - 1 + M_{l^*l^*}) + 4(\mu - 1 + M_{1l^*}) - \lambda = 0 \\
\lambda (\alpha + \beta - \gamma - \mu) = 0 \\
\alpha + \beta \leq \gamma + \mu \\
\lambda \geq 0. 
\end{cases}
\end{align*}
Complementarity implies that either $\alpha + \beta = \gamma + \mu$ or $\lambda = 0$. The latter case leads to the KKT-point $(\alpha, \beta, \gamma, \mu) = \left(M_{il}, M_{jl}, M_{ij}, \frac{1}{6}M_{ll} + \frac{1}{3}M_{1l} - \frac{1}{6}M_{l^*l^*} - \frac{1}{3}M_{1l^*} + \frac{1}{2} \right)$, which is optimal if $M_{il} + M_{jl} \leq M_{ij} + \frac{1}{6}M_{ll} + \frac{1}{3}M_{1l} - \frac{1}{6}M_{l^*l^*} - \frac{1}{3}M_{1l^*} + \frac{1}{2}$. 

Now assume that $\lambda \neq 0$. Then $\alpha + \beta = \gamma + \mu$.  The first four equalities of the KKT system can be rewritten as: 
\begin{align*}
    \alpha &= M_{il} - \frac{1}{4}\lambda, \quad \beta = M_{jl} - \frac{1}{4}\lambda, \quad \gamma = M_{ij} + \frac{1}{4}\lambda \\
    \mu &= \frac{1}{6}M_{ll} + \frac{1}{3}M_{1l} - \frac{1}{6}M_{l^*l^*} - \frac{1}{3}M_{1l^*} + \frac{1}{2} + \frac{1}{12}\lambda
\end{align*}
Substitution into $\alpha + \beta = \gamma + \mu$ yields
\begin{align*}
    M_{il} & - \frac{1}{4}\lambda + M_{jl} - \frac{1}{4}\lambda   = M_{ij} + \frac{1}{4}\lambda +  \frac{1}{6}M_{ll} + \frac{1}{3}M_{1l} - \frac{1}{6}M_{l^*l^*} - \frac{1}{3}M_{1l^*} + \frac{1}{2} + \frac{1}{12}\lambda \\
    \Longleftrightarrow \quad \lambda & = \frac{12}{10}\left( M_{il} + M_{jl} - M_{ij} - \frac{1}{6}M_{ll} - \frac{1}{3}M_{1l} + \frac{1}{6}M_{l^*l^*} + \frac{1}{3}M_{1l^*} - \frac{1}{2}\right). 
\end{align*}
Substitution of this expression for the Lagrange multiplier into the remaining equalities provides the optimal values for $\alpha, \beta, \gamma$ and $\mu$: 
\begin{align*}
    \alpha & = \frac{7}{10} M_{il} - \frac{3}{10} M_{jl} + \frac{3}{10} M_{ij} + \frac{1}{20}M_{ll} + \frac{1}{10}M_{1l} - \frac{1}{20}M_{l^*l^*} - \frac{1}{10}M_{1l^*}  + \frac{3}{20} \\
    \beta & = -\frac{3}{10} M_{il} + \frac{7}{10} M_{jl} + \frac{3}{10} M_{ij} + \frac{1}{20}M_{ll} + \frac{1}{10}M_{1l} - \frac{1}{20}M_{l^*l^*} - \frac{1}{10}M_{1l^*}  + \frac{3}{20} \\
    \gamma & = \frac{3}{10} M_{il} + \frac{3}{10} M_{jl} + \frac{7}{10} M_{ij} - \frac{1}{20}M_{ll} - \frac{1}{10}M_{1l} + \frac{1}{20}M_{l^*l^*} + \frac{1}{10}M_{1l^*}  - \frac{3}{20} \\
    \mu & = \frac{1}{10} M_{il} + \frac{1}{10} M_{jl} - \frac{1}{10} M_{ij} + \frac{3}{20}M_{ll} + \frac{3}{10}M_{1l} - \frac{3}{20}M_{l^*l^*} - \frac{3}{10}M_{1l^*}  + \frac{9}{20}. 
\end{align*}
Setting $\hat{M}_{il} = \hat{M}_{li} = \alpha$, $\hat{M}_{jl} = \hat{M}_{lj} = \beta$, $\hat{M}_{ij} = \hat{M}_{ji} = \gamma$, $\hat{M}_{ll} = \hat{M}_{1l} = \hat{M}_{l1} = \mu$ and $\hat{M}_{l^*l^*} = \hat{M}_{1l^*} = \hat{M}_{l^*1}= 1 - \mu$ gives the final result. \qedsymbol

\section{Additional numerical results}\label{App:numericalresults}
In this section we report additional numerical results. 
Table~\ref{tab:LargeGraph_GU_3n_mixcuts=1_tol=10-3} serves to evaluate a quality of the DNN relaxation \eqref{sdp-p2} with additional cuts for  large graphs obtained after adding at most $3n$ cuts in each outer while-loop of Algorithm \ref{Alg:CP_ADMM}.
In Table~\ref{tab:LargeGraph_GU_5n_mixcuts=1_tol10-3}, Section~\ref{sect:numerics} we report the results when the number of added cuts in each outer while-loop   is at most $5n$.
Our numerical  results show that lower bounds  might significantly improve when adding more cuts. Therefore, our final choice for adding cuts for large graphs is $5n$.

\medskip
We furthermore give in Table~\ref{tab:PlanarGraph_3n_mixcuts=1_tol=10-3} additional numerical results for the DNN relaxation \eqref{sdp-p2} with additional cuts and $k=2$ for (rather small) instances from the literature. All these instances have been considered in~\cite{Hager2013AnEA}.
The first group of instances are grid graphs of Brunetta et al.~\cite{brunetta1997branch}. These graphs are as follows.
\begin{itemize}
\item  Planar grid instances: To represent instances of equicut on planar grid graphs we assign a weight from 1 to 10, drawn from a uniform distribution, to the edges of a $h \times k$ planar grid, and a 0 weight to the other edges. The names of those graphs are formed by the size followed by the letter `g'.
\item Toroidal grid instances: Same as planar grid instances but for toroidal grids. The names of those graphs are formed by the size followed by the letter `t'.
\item  Mixed grid instances: These are dense instances with all edges having a nonzero weight. The edges of a planar grid receive weights from 10 to 100 uniformly generated and all the other edges a weight from 1 to 10, also uniformly generated. The names of those graphs are formed by the size followed by the letter `m'.
\end{itemize}
The second group of instances are randomly generated graphs from~\cite{Hager2013AnEA}: for a fixed density, the edges are assigned integer weights uniformly drawn from [1,10]. Graphs' names begin with `v', `t', `q', `c', and `s'. 

Table \ref{tab:PlanarGraph_3n_mixcuts=1_tol=10-3} also includes results on instances constructed with de~Bruijn networks by~\cite{Hager2013AnEA}, of which the original data arise in applications related to parallel computer architecture \cite{collins1992vlsi,feldmann1997better}. Graphs' names begin with `db'.
Finally, we test  some instances from finite element meshes from~\cite{Hager2013AnEA}, graphs' names begin with `m'.

There are 64 instances in Table \ref{tab:PlanarGraph_3n_mixcuts=1_tol=10-3},  and we prove optimality for 53 instances. The longest time required to compute a lower bound is 2.5 minutes, and computation of upper bounds is negligible.

\begin{landscape}
	\pgfplotstableread[col sep=comma]{largeGraphSparse3n_datafile1.tex}\data
	
	\pgfplotstabletypeset[
	font=\footnotesize,
	multicolumn names, 
	columns={graph_string,n,k,UB,round_DNN_lb,DNN_clocktime,round_cuts_lb,cuts_clocktime,imp_lb_round,TotalTriCuts,TotalCliqueCuts, iter, cut_add_occasions_total},
	fixed zerofill,
	fixed,
 	columns/graph_string/.style={string type,
		column name={graph},
	},
	columns/n/.style={int detect,
		column name={$n$},
	},  
	columns/k/.style={int detect,
		column name={$k$},
	},
	columns/UB/.style={int detect,
		column name={ub},	
	    assign column name/.style={
		/pgfplots/table/column name={\multicolumn{1}{c|}{##1}}}
	},
	columns/round_DNN_lb/.style={int detect,
	column name=lb,
	},
	columns/DNN_clocktime/.style={precision=2, 
    column name=clocktime,
		assign column name/.style={
		/pgfplots/table/column name={\multicolumn{1}{c|}{##1}}}
	},
	columns/round_cuts_lb/.style={int detect,
	column name=lb ,
	},
	columns/cuts_clocktime/.style={precision=2,
	column name=clocktime,
	},
	columns/imp_lb_round/.style={preproc/expr={100*##1},
	column name=imp.,
	dec sep align,
	},
	columns/TotalTriCuts/.style={int detect,
	column name=\# tri. cuts,
	},
	columns/TotalCliqueCuts/.style={int detect,
	column name=\# ind. set cuts,
 	},
 	columns/iter/.style={int detect,
	column name=\# iter.,
	},
 	columns/cut_add_occasions_total/.style={int detect,
	column name=\# outer loops,
	},
	columns addvline/.style={
	every col no #1/.style={
	column type/.add={}{|}}
    },
	columns addvline/.list={3,5},
	    column type=r,
	every head row/.style={
	before row={
	\caption{Large $G$ and $U$ graphs from~\cite{johnson1989optimization} (Maxineq=$3n$) } \label{tab:LargeGraph_GU_3n_mixcuts=1_tol=10-3}\\
	\toprule
	&  &  &  & \multicolumn{2}{c|}{DNN}& \multicolumn{7}{c}{DNN+cuts}\\
	},
	after row={
	&  &  &  &  (rounded) & \multicolumn{1}{c|}{(\si{\second})}  & (rounded) &   \multicolumn{1}{c}{(\si{\second})} & \multicolumn{1}{c}{\%} &  &  & &
	\endfirsthead
	\midrule}
	},
	every last row/.style={after row=\bottomrule}
	]{\data} 
	
	\pgfplotstableread[col sep=comma]{Hager_datafile3.tex}\data
	
	\pgfplotstabletypeset[
	font=\footnotesize,
	multicolumn names, 
	columns={graph,n,k,UB,round_DNN_lb,DNN_clocktime,round_cuts_lb,cuts_clocktime,imp_lb_round,TotalTriCuts,TotalCliqueCuts, iter, cut_add_occasions_total},
	fixed zerofill,
	fixed,
 	columns/graph/.style={verb string type,
		column name={graph},
	},
	skip rows between index={64}{67},
	columns/n/.style={int detect,
		column name={$n$},
	},  
	columns/k/.style={int detect,
		column name={$k$},
	},
	columns/UB/.style={int detect,
		column name={ub},	
	    assign column name/.style={
		/pgfplots/table/column name={\multicolumn{1}{c|}{##1}}}
	},
	columns/round_DNN_lb/.style={int detect,    
	column name=lb,
	},
	columns/DNN_clocktime/.style={precision=2,
    column name=clocktime,
 	assign column name/.style={
 		/pgfplots/table/column name={\multicolumn{1}{c|}{##1}}}
 	},
	columns/round_cuts_lb/.style={int detect,
	column name=lb,
	},
	columns/cuts_clocktime/.style={precision=2,
	column name=clocktime,
	},
	columns/imp_lb_round/.style={preproc/expr={100*##1},
	column name=imp.,
	},
	columns/TotalTriCuts/.style={int detect,
	column name=\# tri. cuts,
	},
	columns/TotalCliqueCuts/.style={int detect,
	column name=\# ind. set cuts,
 	},
 	columns/iter/.style={int detect,
	column name=\# iter.,
	},
 	columns/cut_add_occasions_total/.style={int detect,
	column name=\# outer loops,
	},
	columns addvline/.style={
	every col no #1/.style={
	column type/.add={}{|}}
    },
    column type=r,
	columns addvline/.list={3,5},
	every head row/.style={
	before row={
	\caption{Graphs considered in the paper of Hager, Phan and Zhang~\cite{Hager2013AnEA} (Maxineq=$3n$) }\label{tab:PlanarGraph_3n_mixcuts=1_tol=10-3}\\
	\toprule
	&  &  &  & \multicolumn{2}{c|}{DNN}& \multicolumn{7}{c}{DNN+cuts}\\
	},
	after row={
	&  &  &  &  (rounded) & \multicolumn{1}{c|}{(\si{\second})}   & (rounded) &  \multicolumn{1}{c}{(\si{\second})}  &  \multicolumn{1}{c}{\%} &  &  & &
	\endfirsthead
	\midrule}
	},
	every last row/.style={after row=\bottomrule}
	]{\data} 
	
\end{landscape}

\end{document}